\documentclass[11pt]{amsart}

\usepackage{amssymb,amsfonts,amsmath}
\usepackage{scalerel}
\usepackage{makecell}
\usepackage{amsthm, stmaryrd}
\usepackage{graphicx}
\usepackage[all,arc]{xy}
\usepackage{hyperref}
\hypersetup{colorlinks,allcolors=black}
\usepackage{enumerate}
\usepackage{bbm}
\usepackage{dsfont}
\usepackage{mathrsfs}
\usepackage{tikz-cd}
\usetikzlibrary{shapes}
\usepackage{import}
\usepackage{float}
\restylefloat{table}
\usepackage{multirow}
\usepackage{pdflscape}
\usepackage{listofitems}
\usepackage{standalone}
\usepackage[T1]{fontenc}
\usepackage[toc,page]{appendix}
\usepackage{diagbox}
\usepackage[margin=1in]{geometry}
\usepackage{mathtools}
\usetikzlibrary{decorations.pathmorphing}
\usepackage{mathrsfs}
\usepackage{cancel}
\usepackage{relsize}
\newtheorem{thm}{Theorem}[section]
\newtheorem*{thm*}{Theorem}
\newtheorem*{metathm*}{Meta Theorem}

\newtheorem*{setup*}{Setup}

\newtheorem{innercustomthm}{Theorem}
\newenvironment{customthm}[1]
  {\renewcommand\theinnercustomthm{#1}\innercustomthm}
  {\endinnercustomthm}

\newtheorem{cor}[thm]{Corollary}
\newtheorem{prop}[thm]{Proposition}
\newtheorem{lem}[thm]{Lemma}

\theoremstyle{definition}
\newtheorem{defn}[thm]{Definition}

\newtheorem{notn}[thm]{Notation}

\newtheorem{rem}[thm]{Remark}

\newtheorem*{thm1.2}{\textrm{Theorem 1.2}}

\theoremstyle{remark}
\newcommand{\Mbar}{\overline{\mathcal{M}}}
\newcommand{\Mps}{\mathfrak{M}}
\newcommand{\Cps}{\mathfrak{C}}

\newcommand{\M}{\mathcal{M}}

\newtheorem{sch}[thm]{Scholium}

\newcommand{\Z}{\mathbb{Z}}
\newcommand{\Q}{\mathcal{Q}}
\newcommand{\QQ}{\mathbb{Q}}
\newcommand{\Qbar}{\overline{\mathcal{Q}}}
\renewcommand{\P}{\mathbb{P}}
\newcommand{\bbS}{\mathbb{S}}
\newcommand{\bbL}{\mathbb{L}}

\newcommand{\ch}{\operatorname{ch}}

\newcommand{\qbinomopt}[3][q]{\genfrac{[}{]}{0pt}{}{#2}{#3}_{#1}}

\newcommand{\cat}[1]{\mathsf{#1}}

\newcommand{\nocontentsline}[3]{}
\let\origcontentsline\addcontentsline
\newcommand\stoptoc{\let\addcontentsline\nocontentsline}
\newcommand\resumetoc{\let\addcontentsline\origcontentsline}

\def\A{\mathbb{A}}

\def\C{\mathbb{C}}

\def\P{\mathbb{P}}

\def\Z{\mathbb{Z}}

\def\calA{\mathcal{A}}
\def\calB{\mathcal{B}}
\def\calC{\mathcal{C}}

\def\calE{\mathcal{E}}

\def\calO{\mathcal{O}}

\def\calQ{\mathcal{Q}}

\def\calS{\mathcal{S}}
\def\calT{\mathcal{T}}

\def\calV{\mathcal{V}}
\def\calW{\mathcal{W}}
\def\calX{\mathcal{X}}
\def\calY{\mathcal{Y}}

\def\M{\mathcal{M}}

\newcommand{\Sym}{\operatorname{Sym}}
\newcommand{\Ind}{\operatorname{Ind}}
\newcommand{\Res}{\operatorname{Res}}
\newcommand{\Spec}{\operatorname{Spec}}
\newcommand{\Exp}{\operatorname{Exp}}

\newcommand{\ve}{\varepsilon}

\newcommand{\Quot}{\mathrm{Quot}}

\newcommand\cycle[2][\,]{%
  \readlist\thecycle{#2}%
  (\foreachitem\i\in\thecycle{\ifnum\icnt=1\else#1\fi\i})%
}

\newcommand{\Gr}{\mathbb{G}}
\newcommand{\GL}{\mathrm{GL}}
\renewcommand{\-}{\text{-}}

\newcommand{\bbA}{\mathbb{A}}

\makeatletter
\let\c@equation\c@thm
\makeatother
\numberwithin{equation}{section}

\graphicspath{ {images/} }

\title{Motivic quasimap wall-crossing for Grassmannians}

\author[S. Kannan]{Siddarth Kannan}\address{Department of Mathematics, Massachusetts Institute of Technology}
\email{\url{spkannan@mit.edu}}
\author[T. Song]{Terry Dekun Song}\address{Department of Pure Mathematics and Mathematical Statistics, University of Cambridge, Cambridge, CB3 0WA}\email{\url{ds2016@cantab.ac.uk}}

\begin{document}
\maketitle

\begin{abstract}
We prove a wall-crossing formula for the Euler characteristics, considered as virtual mixed Hodge structures, of moduli spaces of $\ve$-stable quasimaps to the Grassmannian $\Gr(r, N)$. For each $\ve > 0$, we define a $\QQ$-algebra automorphism of the ring of symmetric functions which takes the generating function for the $\bbS_n$-equivariant Euler characteristics of the moduli spaces of stable maps $\Mbar_{g, n}(\Gr(r, N), d)$ to the corresponding generating function for Toda's moduli spaces of $\ve$-stable quasimaps $\Qbar_{g, n}^{\ve}(\Gr(r, N), d)$. The automorphism is given by explicit $q$-deformations of the power sum symmetric functions. The $\ve \to 0$ limit of our formula exchanges the spaces of stable maps and the Marian--Oprea--Pandharipande moduli spaces of stable quotients. Our proof uses the geometry of relative Quot schemes to relate the quasimap spaces to moduli spaces of \textit{weighted} stable maps, for which we obtain wall-crossing formulas via symmetric function theory. 
\end{abstract}

\section{Introduction}
For a GIT quotient $V\sslash G\subset [V/G],$ both the moduli space of stable maps $\Mbar_{g, n}(V\sslash G;\beta)$ and the moduli space of $\ve$-stable quasimaps $\Qbar_{g, n}^{\ve}(V\sslash G;\beta)$ constructed in \cite{quasimaps_gitquotients} determine enumerative theories of curves in $V\sslash G$, leading to Gromov--Witten invariants and quasimap invariants, respectively. For a large class of targets, the corresponding generating functions for these invariants are known to be equivalent via explicit wall-crossing formulas \cite{CFK14, CFK17, CFK20, Zhou2021, CJR}.

In this article we establish topological and motivic incarnations of the enumerative wall-crossing formulas, with
\[V\sslash G =  \Gr(r, N)\subset [\mathrm{Hom}(\C^{r}, \C^N)/\GL_{r}(\C)], \]
where $\Gr(r, N)$ denotes the Grassmannian of $r$-planes in $\C^N$. In this case, the moduli stacks of $\ve$-stable quasimaps were first constructed by Toda \cite{Toda}, preceding the construction for general GIT quotients in \cite{quasimaps_gitquotients}. The $\ve \to 0^+$ limit of the space of $\ve$-stable quasimaps recovers the moduli space $\Qbar_{g, n}(\Gr(r, N), d)$ of stable quotients constructed by Marian--Oprea--Pandharipande \cite{MOP}, while taking $\ve > 1$ recovers the moduli space of stable maps.

Our main theorem concerns the Euler characteristic of $\Qbar_{g, n}^{\ve}(\Gr(r, N), d)$ as a virtual mixed Hodge structure. As $\ve$ varies, we find that the generating functions for the Euler characteristics are equivalent under explicit invertible transformations. To state our formula precisely, let $\bbS_n$ denote the symmetric group on $n$ letters, and let $\cat{MHS}_{\bbS_n}$ denote the category of $\bbS_n$-representations in the category $\cat{MHS}$ of mixed Hodge structures over $\QQ$. If we write $K_0(\cat{MHS}_{\bbS_n})$ for the Grothendieck group, there is a natural isomorphism $K_0(\cat{MHS}_{\bbS_n}) \otimes \QQ \cong K_0(\cat{MHS}) \otimes \mathrm{Cl}_{\QQ}(\bbS_n)$ where $\mathrm{Cl}_{\QQ}(\bbS_n)$ is the vector space of $\QQ$-valued class functions on $\bbS_n$. For an $\bbS_n$-variety $X$, define its \textit{$\bbS_n$-equivariant Serre characteristic} by 
\[ \cat{e}^{\bbS_n}(X) := \sum_{i} (-1)^i[H^i_c(X;\QQ)] \in K_0(\cat{MHS}) \otimes \mathrm{Cl}_{\QQ}(\bbS_n). \]
 Note that $\cat{e}^{\bbS_n}(X)$ specializes to the ordinary Serre characteristic $\cat{e}(X) := \cat{e}^{\bbS_0}(X)$, which in turn specializes to the $E$-polynomial of $X$, in the sense of \cite[Definition 2.1.4]{HR-V}.

To state our formula, we use the identification
\[ \prod_{n \geq 0} K_0(\cat{MHS}) \otimes \mathrm{Cl}_{\QQ}(\bbS_n)  \cong K_0(\cat{MHS}) \otimes \Lambda\]
where \(\Lambda =  \QQ[\![p_1, p_2, \ldots]\!]\) is the ring of symmetric functions over $\QQ$, and $p_i$ is the $i$th power sum symmetric function. We define generating functions
\[ \overline{\cat{M}}_{g, r, N} := \sum_{n, d} \cat{e}^{\bbS_n}(\Mbar_{g, n}(\Gr(r, N), d))q^d \quad \mbox{and} \quad \overline{\cat{Q}}^{\ve}_{g, r, N} := \sum_{n, d} \cat{e}^{\bbS_n}(\Qbar_{g, n}^{\ve}(\Gr(r, N), d))q^d \]
as elements of $K_0(\cat{MHS}) \otimes \Lambda [\![q]\!]$. We write $\cat{M}_{g, r, N}$ and $\cat{Q}_{g, r, N}^{\ve}$ for the analogous generating functions for the loci of stable maps and quasimaps with smooth domain curves. We also write \[\overline{\cat{Q}}_{g, r, N}^{0} : = \sum_{n, d} \cat{e}^{\bbS_n}(\Qbar_{g, n}(\Gr(r, N), d))q^d \quad \mbox{and} \quad \cat{Q}_{g, r, N}^0 := \sum_{n, d} \cat{e}^{\bbS_n}(\Q_{g, n}(\Gr(r, N), d))q^d\] for the corresponding generating functions for moduli spaces of stable quotients; these are obtained as the $\ve \to 0^+$ limits\footnote{To make this precise, we refer to the end of Definition \ref{defn:vequasimaps} for the relationship between stable quotients and $\ve$-stable quasimaps.} of $\overline{\cat{Q}}_{g, r, N}^\ve$ and $\cat{Q}_{g, r, N}^\ve$, respectively.

Our formula is given by certain $q$-deformations of the power sums, as we now explain. Recall the \textit{$q$-binomial coefficient}
\begin{equation}\label{eqn:q_binom_defn}
\qbinomopt[q]{N}{r} = \frac{(1 - q^{N})(1-q^{N - 1}) \cdots (1 - q^{N - r + 1})}{(1 - q^{r})(1- q^{r - 1}) \cdots (1 - q)}.
\end{equation}
In general, for an indeterminate $t$, we write $\qbinomopt[t]{N}{r}$ for the evaluation of the $q$-binomial coefficient at $q = t$. For example, we have the well-known formula
\[ \cat{e}(\Gr(r, N)) = \qbinomopt[\bbL]{N}{r} \]
where $\bbL = [H^2(\P^1;\QQ)] \in K_0(\cat{MHS)}$.

For fixed $\ve > 0$, we let $M_\ve = \lfloor1/\ve \rfloor$ be the largest integer bounded above by $1/\ve$. We define a $\QQ$-algebra automorphism \(\boldsymbol{B}_{r, \ve}: K_0(\cat{MHS}) \otimes \Lambda[\![q]\!] \to K_0(\cat{MHS}) \otimes \Lambda[\![q]\!]\) by
\begin{equation}\label{eqn:Bre}
    \boldsymbol{B}_{r, \ve}: p_j \mapsto p_j + \sum_{k = 1}^{M_{\ve}} \qbinomopt[\bbL^j]{r + k - 1}{k} \ q^{jk}
\end{equation}
We also define $\boldsymbol{B}_{r,0}$ by
\begin{equation}\label{eqn:Br0}
    \boldsymbol{B}_{r, 0}: p_j \mapsto p_j +  \frac{1}{\prod_{i = 0}^{r-1}(1 - (\bbL^{i}q)^{j})} - 1
\end{equation}
which is in fact the $\ve \to 0^+$ limit of $\boldsymbol{B}_{r, \ve}$. Our main theorem states that up to \textit{plethysm} with certain genus-zero contributions, the stable map and quasimap generating functions are related by $\boldsymbol{B}_{r, \ve}$. 

\begin{customthm}{A}\label{thm:main}
For any genus $g \geq 1$ and rational number $\ve \geq 0$, we have
\[\boldsymbol{B}_{r, \ve}\left[ \overline{\cat{M}}_{g, r, N} \circ \left(p_1 - \frac{\partial}{\partial p_1}\frac{\cat{M}_{0, r, N}}{\cat{e}(\Gr(r, N))}\right)\right] = \overline{\cat{Q}}_{g, r, N}^{\ve} \circ \left(p_1 - \frac{\partial}{\partial p_1}\frac{\cat{Q}_{0, r, N}^{\ve}}{\cat{e}(\Gr(r, N))}\right),\]
where $\circ$ denotes plethysm of symmetric functions.
\end{customthm}
As we recall in \S\ref{sec:symmetric_functions}, plethysm is an algorithmically computable operation on symmetric functions. The plethysms involved in Theorem \ref{thm:main} are invertible, and Bagnarol \cite{Bagnarol} has computed $\cat{M}_{0, r, N}$. We address the genus-zero case and determine $\cat{Q}_{0, r, N}^{\ve}$ explicitly in §\ref{sec:g0maps}. In particular, Theorem \ref{thm:main} gives a computable and invertible transformation between $\overline{\cat{M}}_{g, r, N}$ and $\overline{\cat{Q}}^{\ve}_{g, r, N}$.

We emphasize that it is very difficult to calculate the Serre characteristic of $\Qbar_{g, n}^{\ve}(\Gr(r, N), d)$ in general, due in part to geometric pathologies: it can have arbitrary singularities \cite{Vakil} and is usually reducible with components of excess dimension. Nevertheless, Theorem \ref{thm:main} reflects strong constraints on how the geometry of the moduli space changes with $\ve$, and in particular implies that the different $\overline{\cat{Q}}_{g, r, N}^{\ve}$ encode essentially equivalent information. Our result is new even when $r = 1$, which corresponds to the target $\P^{N - 1}$. In this case the formulas for $\boldsymbol{B}_{r, \ve}$ and $\boldsymbol{B}_{r, 0}$ simplify considerably. We obtain an alternative wall-crossing formula for $\P^{N-1}$ in Corollary \ref{cor:alternative_wall_crossing} by exploiting the existence of contraction morphisms
\[ \Mbar_{g, n}(\P^{N - 1}, d) \to \Qbar_{g, n}^{\ve}(\P^{
N-1
}, d), \]
which do not exist when $r > 1$.

\begin{table}[h]
\centering
\renewcommand{\arraystretch}{1.4}
\setlength{\tabcolsep}{10pt}
\begin{tabular}{|c|c|c|c|}
\hline
$d$ & $\chi(\Mbar_{1,0}(\P^{N-1},d))$ &$\chi(\Qbar_{1,0}^{1/2}(\P^{N-1},d))$ & $\chi(\Qbar_{1,0}(\P^{N-1},d))$ \\
\hline
$1$
& $4\binom{N}{2}$ & $2\binom{N}{1}$ 
& $2\binom{N}{1}$ \\
\hline
$2$
& $17\binom{N}{2}+24\binom{N}{3}$ & $4\binom{N}{1} + 2\binom{N}{2}$
& $4\binom{N}{1} + \binom{N}{2}$ \\
\hline
$3$
& $55\binom{N}{2}+247\binom{N}{3}+216\binom{N}{4}$ & $8\binom{N}{1}+3\binom{N}{2}+\binom{N}{3}$
& $6\binom{N}{1}+3\binom{N}{2}+\binom{N}{3}$ \\
\hline
\end{tabular}
\caption{The topological Euler characteristic of the moduli spaces of stable maps and $\ve$-stable quasimaps to $\P^{N-1}$ (the $r=1$ case) in genus one, for $\ve = 1/2$ and $\ve \to 0^+$. 
}
\label{table:data}
\end{table}

\begingroup
\renewcommand{\arraystretch}{1.5}
\setlength{\tabcolsep}{12pt}

\begin{table}[htbp]
\centering
\caption{The topological Euler characteristic of the moduli space of stable quotients $\Qbar_{1, 0}(\Gr(r, N), d)$, for small values of $r$ and $d$.}
\label{table:quotients_genus_one}
\begin{tabular}{|c|c|c|c|}
\hline
$d$ & $\chi(\Qbar_{1, 0}(\Gr(2, N), d))$ & $\chi(\Qbar_{1, 0}(\Gr(3, N), d))$ & $\chi(\Qbar_{1, 0}(\Gr(4, N), d))$\\
\hline
$1$ & $4\binom{N}{2}$ & $6\binom{N}{3}$ & $8\binom{N}{4}$ \\
\hline
$2$ & $12\binom{N}{2}+3\binom{N}{3}$ & $24\binom{N}{3}+6\binom{N}{4}$ & $40\binom{N}{4}+10\binom{N}{5}$ \\
\hline
$3$ & $28\binom{N}{2}+16\binom{N}{3}+4\binom{N}{4}$ & $76\binom{N}{3}+46\binom{N}{4}+10\binom{N}{5}$ & $160\binom{N}{4}+100\binom{N}{5}+20\binom{N}{6}$ \\
\hline
\end{tabular}
\label{table:chiQ10bar}
\end{table}

\endgroup

The operators $\boldsymbol{B}_{r, \ve}$ can be written in terms of plethysm as well. If we formally set $M_0 = \infty$, then for any $f \in K_0(\cat{MHS}) \otimes \Lambda[\![q]\!]$ and $\ve \geq 0$ we have
\begin{equation}
    \boldsymbol{B}_{r, \ve}(f) = f \circ \left(p_1 + \sum_{k = 1}^{M_\ve} \qbinomopt[\bbL]{r + k - 1}{k} q^k \right).
\end{equation} 
This formula has a geometric interpretation: since
\[ \qbinomopt[\bbL]{r + k - 1}{k} = \cat{e}(\Sym^{k}(\P^{r - 1})), \]
each operator $\boldsymbol{B}_{r,\ve}$ is plethysm with a truncation of the motivic zeta function of $\P^{r - 1}$. See also Remark \ref{rem:punctualquot}.

In \S\ref{sec:genus_one}, we explain how to calculate the $\bbS_n$-equivariant topological Euler characteristic of the moduli space $\Mbar_{1, n}(\Gr(r, N), d)$ (Proposition \ref{prop:chi_genus_one}), extending our previous work on the $r = 1$ case \cite{ks-genus1}. The calculation amounts to counting walks in the $1$-skeleton of the hypersimplex $\Delta(r, N)$, which encodes the combinatorics of a generic $\C^\star$-action on $\Gr(r, N)$. Via Theorem \ref{thm:main}, Proposition \ref{prop:chi_genus_one} determines the topological Euler characteristic of $\Qbar_{1, n}^{\ve}(\Gr(r, N), d)$ for arbitrary $\ve,n,r,N$, and $d$, see Corollary \ref{cor:top_chi_genus_one} and the preceding discussion. In principle, $\C^\star$-localization \cite[\S 5.1]{Toda} and the techniques of \cite{kspp} give graph-sum formulas for the topological Euler characteristic of $\Qbar_{g, n}^{\ve}(\Gr(r, N), d)$ in arbitrary genus; the point of \S \ref{sec:genus_one} is that the graph sums can be solved exactly in genus one. Sample calculations are provided in Tables \ref{table:data} and \ref{table:quotients_genus_one}.

\subsection{Outline of the proof}

The proof of Theorem \ref{thm:main} can be summarized as the following three main steps:
\begin{enumerate}
    \item reduce to comparing stable maps and $\ve$-stable quasimaps without rational tails;
    \item relate $\ve$-stable quasimaps to moduli spaces of weighted stable maps via relative Quot schemes;
    \item apply symmetric function theory to wall-cross between moduli spaces of weighted stable maps.
\end{enumerate}

Step (1) is handled by Lemma \ref{lem:nrtplethysm}, which  separates the contributions of maps without rational tails and genus-zero maps via plethysm of symmetric functions. This technique is an important step in several previous works on motivic invariants of moduli spaces of curves and maps \cite{semiclassicalremark, PaganiTommasi, ks-genus1, kspp}.

The main geometric input to the formula is in step (2), where we stratify the moduli space of $\ve$-stable quasimaps without rational tails into certain relative Quot schemes (Lemma \ref{lem:quot_strata_stacky}) of zero-dimensional quotient sheaves. Their motives are in turn expressed in Theorem \ref{thm:relquot-to-weightedstablemaps} in terms of moduli spaces of weighted stable maps constructed by Alexeev--Guy \cite{AlexeevGuy} and Bayer--Manin \cite{BayerManin}. We highlight two aspects of our calculations in this step.
\begin{enumerate}
  \item We determine the motive of the \textit{relative} Quot scheme of a family of smooth curves (Proposition \ref{prop:relquot}), building on previous works \cite{Bif89, BFP20, Ric20} for a curve over $\Spec \C.$ The formula passes through an independence result on the framing vector bundle in the relative setting (Lemma \ref{lem:indep-Quot}) as well as a stratification of the Quot scheme by iterated vector bundles over relative symmetric powers of the curve, akin to the complete filt schemes studied by Mochizuki \cite{Mochizuki}.
  \item The relative Quot scheme is taken over the Artin stack of \textit{prestable} maps without rational tails. We then impose a stability condition on the relative Quot scheme that matches both $\ve$-quasimap stability and weighted stable maps stability. Working with prestable maps is necessary to allow rational bridges that are stabilized by the presence of torsion points.
\end{enumerate}

We hence reduce to comparing motivic invariants of these quotients of moduli spaces of weighted stable curves, which is taken up in step (3). The wall-crossing formulas for weighted stable maps are encoded by fundamental transformations of symmetric functions (Corollary \ref{cor:serre_char_collisions}). Combining these transformations with the relative Quot scheme formula leads to the operators $\boldsymbol{B}_{r, \ve}$ defined earlier. 

Geometrically, the operators $\boldsymbol{B}_{r, \ve}$ and their $\ve\to 0^+$ limit transform weighted marked points with permutation actions to zero-dimensional quotients supported at the divisor underlying the markings. The moduli of these quotients is encoded in the powers of $\bbL$ appearing in the definition of $\boldsymbol{B}_{r, \ve}$. The powers of $q$ which appear in $\boldsymbol{B}_{r, \ve}$ have to do with tracking collisions of marked points and taking quotients by symmetric group actions, generalizing previous work of Kannan--Serpente--Yun \cite{hodgehl} on Hassett's moduli spaces of weighted stable curves \cite{Hassett}.

The final step in making Theorem \ref{thm:main} effective is to explain how to calculate and invert the plethysms appearing in the formula, which is handled by work of Getzler--Pandharipande \cite{GetzlerPandharipande} on Serre characteristics of moduli spaces of genus-zero stable maps to $\P^r$ and its extension to Grassmannians by Bagnarol \cite{Bagnarol}.

\subsection{Context and related work}
This article begins a program to compare motivic and topological invariants of the moduli spaces of stable maps and quasimaps, in line with the extensive literature on wall-crossing formulas for quasimap invariants \cite{CFK14, CFK17, CFK20, Zhou2021}, which serves as a main inspiration. While the techniques and results of this work are not \textit{a priori} related to the aforementioned wall-crossing formulas for quasimap invariants, both calculations involve the combinatorics of rational tails and wall-crossing via weighted markings. Theorem \ref{thm:main} has a similar shape to the enumerative wall-crossing formulas: see for example \cite[Corollary 1.11.3]{Zhou2021}. It would be interesting to see whether the techniques of this work have any application to enumerative invariants. We expect that relative versions of previous results on the cohomology of tautological bundles over Quot schemes \cite{TautQuot1, TautQuot2} would be relevant to this question.

 The study of the cohomology of moduli spaces of stable maps to Grassmannians and flag varieties has attracted significant interest \cite{Manin98, LopezMartin, opreaflag, Zhuang2018, Bagnarol}. Especially important for us is work of Getzler and Pandharipande introducing symmetric functions in this setting in order to calculate the $\bbS_n$-equivariant Serre characteristic of $\Mbar_{0, n}(\P^{r}, d)$ \cite{GetzlerPandharipande}. Bagnarol \cite{Bagnarol} built on their work to determine the $\bbS_n$-equivariant Serre characteristic of $\Mbar_{0,n}(\Gr(r,N),d)$. His work also expresses the Serre characteristic of the interior $\M_{0,n}(\Gr(r,N),d)$ in terms of Quot schemes (both of rank-$r$ sheaves and of points) on $\P^1.$ The comparison with the rank-$r$ Quot scheme has been an effective technique in earlier works \cite{Stromme, BDW, BERTRAM1997, opreaflag} on mapping spaces to the Grassmannian, while Popa--Roth \cite{PopaRoth} used stable maps to Grassmannians to probe the geometry of Quot schemes. Much less is known about motivic invariants of moduli spaces of quasimaps, though Cooper \cite{cooper} computed the Poincar\'e polynomial of $\Qbar_{1,0}(\Gr(1, 1), d)$.

As discussed above, the relative Quot scheme of $0$-dimensional quotients on a family of smooth curves plays a fundamental role in the present work. Our calculation of its motive closely follows the approach in \cite{BFP20, Ric20, MonavariRicolfi} on the ordinary Quot scheme of $0$-dimensional quotients on a variety. The parallel study of the motives of hyperquot schemes of curves \cite{Chen2001, Monavari2022, MonavariRicolfi2025} may be useful for extending the wall-crossing formulas in this work to generalized flag varieties.

The operators $\boldsymbol{B}_{r, \ve}$ and $\boldsymbol{B}_{r, 0}$ build on our previous works \cite{VirtualHodge, KannanPic} which calculate the Serre characteristics of moduli spaces of maps and line bundles by using symmetric function theory to reduce to the Serre characteristics of $\M_{g,n}$. There is an extensive body of literature on applications of symmetric function theory to Euler characteristics of moduli spaces of curves; see e.g. \cite{GetzlerGenusZero, GetzlerKapranov, GetzlerSemiClassical, semiclassicalremark, BergstromMinabe1, BergstromMinabe2, ks-genus1, kspp} for an incomplete list of references.

\subsection{Plan of the paper} We begin by recalling the definitions of the relevant moduli problems in \S\ref{sec:moduli_definitions}. In \S\ref{sec:symmetric_functions}, we set up the symmetric function theory which will be used throughout the paper, and apply it to moduli spaces of weighted stable maps. In \S\ref{sec:relative_quot}, we study motives of relative Quot schemes related to the geometry of the moduli space of quasimaps, and establish our wall-crossing formula for the no-rational-tails loci. We study the genus-zero case and finish the proof of Theorem \ref{thm:main} in \S\ref{sec:g0maps}. In \S \ref{sec:genus_one}, we conclude the paper by explaining how to calculate topological Euler characteristics of spaces of the quasimap moduli spaces explicitly in genus one.

\subsection*{Acknowledgments}
We are grateful to Dhruv Ranganathan for encouraging us to pursue this direction and for helpful comments on a previous draft of this article. SK is supported by NSF DMS-2401850. TS is supported by Cambridge Trust international scholarship and BICMR.

\section{Moduli spaces of maps}\label{sec:moduli_definitions}
We continue to let $\Gr(r, N)$ denote the Grassmannian of $r$-planes in $\C^N$. We recall moduli spaces of $\ve$-stable quasimaps to $\Gr(r, N)$ as constructed by Toda \cite{Toda} and Ciocan-Fontanine--Maulik--Kim \cite{quasimaps_gitquotients}, with the Kontsevich moduli space of stable maps \cite{KontsevichTorus} and the Marian--Oprea--Pandharipande moduli space of stable quotients \cite{MOP} as special cases. We will also recall the moduli spaces of \textit{weighted} stable maps constructed by Alexeev--Guy \cite{AlexeevGuy} and Bayer--Manin \cite{BayerManin}, following earlier work of Hassett \cite{Hassett}. Ultimately, we will stratify the moduli spaces of stable quasimaps by iterated vector bundles over strata of finite quotients of moduli spaces of weighted stable maps.
 \subsection{Quasimaps} Let $\ve$ be a positive rational number. In \cite{quasimaps_gitquotients}, a moduli space of $\ve$-stable quasimaps to a large class of GIT quotients is constructed. We study their construction in the special case of the Grassmannian $\Gr(r, N)$, via the GIT quotient presentation
\[\Gr(r, N) = \mathrm{Hom}(\C^r, \C^N) \sslash \mathrm{GL}_r(\C). \] In this case, the moduli space was first constructed by Toda \cite{Toda}.
\begin{defn} \label{defn:vequasimaps}
     The moduli stack
$\overline{\Q}_{g, n}^{\ve}(\Gr(r, N), d)) $
of $\ve$-stable quasimaps to the Grassmannian $\Gr(r, N)$ parametrizes families of tuples
\[ [(C, p_1, \ldots, p_n), 0 \to \calV \to \calO^N_C \to \calQ \to 0] \]
where
\begin{enumerate}
    \item $(C, p_1, \ldots, p_n)$ is an $n$-pointed prestable curve of genus $g$;
    \item $\calV$ is a locally free sheaf of rank $r$ on $C$ of degree $-d$;
    \item the quotient sheaf $\calQ$ is locally free at the nodes and markings of $C$;
    \item the $\QQ$-line bundle
    \[ \omega_C(\sum_i p_i) \otimes (\wedge^{r} \calV^*)^{\otimes \ve} \]
    is ample, where $\calV^*$ denotes the dual of $\calV$.
\end{enumerate}

In simple terms, the stability condition (4) can be summarized as follows:
\begin{itemize}
    \item for every $p \in C,$ the length $\ell(p)$ of the torsion part of $\calQ$ at $p$ must satisfy $\ve \cdot \ell(p) \leq 1$,
    \item for any irreducible component $T$ of $C$, if we let $n_T$ denote the number of special points (markings and preimages of nodes) which are contained in the normalization of $T$, then we have
    \[ 2g(T) - 2 + n_T + \ve\deg(\calV|_{T}) > 0. \]
\end{itemize}

Let $\Q^{\ve}_{g,n}(\Gr(r,N),d)\subset\Qbar^{\ve}_{g,n}(\Gr(r,N),d)$ be the locus of quasimaps with smooth domain curves.
\end{defn}

From the definition, we see that the spaces $\overline{\Q}_{g, n}^{\ve}(\Gr(r, N), d))$ are isomorphic for all ${\ve\in (\frac{1}{k+1}, \frac{1}{k})}.$ This class of moduli spaces interpolates between familiar moduli stacks:
\begin{enumerate}
    \item the Kontsevich moduli space $\Mbar_{g,n}(\Gr(r, N),d)$ of stable maps to $\Gr(r,N)$ coincides with $\Qbar_{g,n}^\ve(\Gr(r,N),d)$ for any $\ve>1$;
    \item the Marian--Oprea--Pandharipande moduli space of stable quotients $\Qbar_{g,n}(\Gr(r, N),d)$ coincides with $\Qbar_{g,n}^\ve(\Gr(r,N),d)$ for any $\ve<1/d.$
\end{enumerate}
The two cases correspond to requiring $\mathcal{Q}$ to be torsion-free and to acquire torsion of arbitrary length, respectively.

\subsection{Stratifications of the quasimap moduli space}
Quasimaps parametrize morphisms from nodal curves to the quotient stack $[\mathrm{Hom}(\C^r, \C^N)/\mathrm{GL}_r(\C)]$ such that a dense open subset of the domain curve -- which contains the nodes and markings -- lands in the GIT quotient $\Gr(r,N).$ The finite subset of the curve landing in the unstable locus is captured by the torsion part of the quotient sheaf $\calQ$ parametrized by $\Qbar^{\ve}_{g,n}(\Gr(r,N),d)$ as defined above. We stratify the moduli stacks by the length of the torsion subsheaf.

\begin{defn}\label{defn:torsionstrat}
    Let \[\Qbar_{g,n}^{\ve, k}(\Gr(r,N),d)\subset \Qbar_{g,n}^{\ve}(\Gr(r,N),d)\] be the locally closed substack parametrizing tuples $[(C,p_1,\dots, p_n), 0\to \calV\to \calO^N_C\to \calQ\to 0]$ as in Definition \ref{defn:vequasimaps}, such that the torsion subsheaf of $\calQ$ has length $k.$ Let $\Q^{\ve,k}_{g,n}(\Gr(r,N),d)$ be the intersection $\Qbar^{\ve,k}_{g,n}(\Gr(r,N),d)\cap \Q^{\ve}_{g,n}(\Gr(r,N),d).$
\end{defn}

\begin{defn}
    Let $C$ be a prestable curve of genus $g$. We say an irreducible component $T \subset C$ is a \textit{rational tail} if $T$ is a smooth rational component and
    \[ |T \cap \overline{C\smallsetminus T}| = 1; \]
    i.e. $T$ meets the rest of $C$ at a single node.
\end{defn}

A prestable genus-$g$ curve $C$ has no rational tails if and only if its genus-decorated dual graph has no proper subgraphs of genus $g$.

\begin{defn}
    For $g>0,$ let \[\Qbar^{nrt,\ve}_{g,n}(\Gr(r,N),d)\subset \Qbar^{\ve}_{g,n}(\Gr(r,N),d)\] be the locally closed substack parametrizing $\ve$-stable quasimaps whose source curves do not have any rational tails.

    We use $\Qbar^{k, nrt,\ve}_{g,n}(\Gr(r,N),d)$ to denote the intersection $\Qbar^{k,\ve}_{g,n}(\Gr(r,N),d)\cap \Qbar^{nrt, \ve}_{g,n}(\Gr(r,N),d).$
\end{defn}

\subsection{Weighted stable maps}
Let $w = (w_1, \ldots, w_n)$ be a vector of non-negative rational weights, with each $w_i \in (0, 1]$. 
\begin{defn}[\cite{AlexeevGuy, BayerManin}]\label{defn:wtstablemaps}
    The moduli space of weighted stable maps
\(\Mbar_{g, w}(\Gr(r, N), d)\)  
parameterizes tuples
\((C, q_1, \ldots, q_n, f)\) 
where 
\begin{enumerate}
    \item $C$ is a prestable curve of genus $g$;
    \item the $q_i \in C$ are smooth marked points;
    \item if \(q_{i_1} = q_{i_2} = \cdots = q_{i_j}\) for some indices $i_1, \ldots, i_j,$ then
    \( w_{i_1} + \cdots + w_{i_j} \leq 1;\)
    \item $f: C \to \Gr(r, N)$ is a degree-$d$ morphism, such that if $f$ contracts an irreducible component $T$ of $C$, then
    \[2g(T) - 2 + \nu_T + \sum_{i:p_i \in T} w_i > 0 \]
    where $\nu_T$ denotes the number of preimages of nodes in the normalization of $T$. 
\end{enumerate}
\end{defn}
Let \( \M_{g, w}(\Gr(r,N), d) \subset \Mbar_{g, w}(\Gr(r, N), d) \) be the locus of maps with smooth domain curves.

A special role will be played by the following weights.

\begin{defn}
    Let $\Mbar_{g, m|\ve^n}(\Gr(r, N), d)$ denote the moduli space $\Mbar_{g,w}(\Gr(r,N),d)$ when \[w = (\underbrace{1, \ldots, 1}_{m}, \underbrace{\ve, \ldots, \ve}_{n}) \] for a rational number $\ve \in (0, 1)$. This moduli space parameterizes $(m + n)$-pointed stable maps, such that only at most $1/\ve$ of the final $n$ points may coincide with one another. When $n = 0$, we recover
\(\Mbar_{g, m|\ve^0}(\Gr(r, N), d) = \Mbar_{g, m}(\Gr(r, N), d). \)
We also set
\[ \Mbar_{g, m|\ve^n}^{nrt}(\Gr(r, N), d) \subset \Mbar_{g, m|\ve^n}(\Gr(r, N), d) \]
for the locus of weighted-stable maps from curves with no rational tails.
\end{defn}

\section{Symmetric functions and weighted stable maps}\label{sec:symmetric_functions}
We now outline the basic properties of symmetric functions that we will use for our calculations. Macdonald's book \cite{Macdonald} is a thorough reference; see also Getzler--Kapranov \cite{GetzlerKapranov}.

Let $\Lambda = \QQ[\![ p_1,p_2,\ldots]\!]$ be the ring of degree-completed symmetric functions over $\QQ.$ Here, $p_i$ is the $i$th power sum symmetric function and has degree $i$. The subspace $\Lambda_n \subset \Lambda$ of homogeneous degree-$n$ symmetric functions is identified with $\mathrm{Cl}_{\QQ}(\bbS_n),$ the space of $\QQ$-valued class functions on $\bbS_n$ via the \textit{Frobenius characteristic}: if $V$ is a finite dimensional $\bbS_n$-representation over $\QQ$, then the Frobenius characteristic of $V$ is defined by
\[ \ch_n(V) := \frac{1}{n!} \sum_{\sigma \in \bbS_n} \mathrm{Tr}(\sigma|V) \prod_{i > 0 } p_i^{\lambda_i(\sigma)}, \]
where $\lambda_i(\sigma)$ is the number of $i$-cycles in $\sigma$.

Let $\cat{MHS}_{\bbS_n}$ be the category of $\bbS_n$-equivariant mixed Hodge structures over $\QQ.$ There is an isomorphism \( K_0(\cat{MHS}_{\bbS_n}) \otimes \QQ \cong K_0(\cat{MHS}) \otimes \Lambda_n \)
by \cite[Theorem 3.2]{GetzlerPreprint}: any $\bbS_n$-equivariant rational mixed Hodge structure $M$ decomposes uniquely as
\[M = \bigoplus_{\lambda \vdash n} M_{\lambda} \otimes \left(E_{\lambda}^{\bigoplus a_{M,\lambda}}\right) \]
for some nonnegative integers $a_{M, \lambda}$, where $M_{\lambda}$ is a rational mixed Hodge structure and $E_{\lambda}$ is the irreducible representation (Specht module) of $\bbS_n$ corresponding to the partition $\lambda$. The isomorphism $K_0(\cat{MHS}_{\bbS_n}) \to K_0(\cat{MHS}) \otimes \Lambda_n$ is determined by
\begin{equation}\label{eqn:frob_char_mhs}
    [M] \mapsto \ch_n(M) := \sum_{\lambda \vdash n} a_{M, \lambda}[M_{\lambda}] \otimes s_{\lambda},
\end{equation}
where $s_{\lambda} \in \Lambda_n$ is the Schur function corresponding to $\lambda$. For a $\bbS_n$-equivariant mixed Hodge structure $M$, we refer to $\ch_n(M)$ as the Frobenius characteristic of $M$.

\subsection{Multisymmetric functions and the coproduct}
For an integer $s > 0$, we have
\[ \Lambda^{\otimes s} = \QQ[\![p_j^{(1)}, p_j^{(2)}, \ldots, p_j^{(s)} | j> 0 ]\!] \]
where
\[p_i^{(k)} = 1 \otimes \cdots \otimes \underbrace{p_i}_{k\text{th entry}} \otimes \cdots \otimes 1 \in \Lambda^{\otimes s}.\]
One basic structure on $\Lambda$ is the coproduct
\[ \Delta: K_0(\cat{MHS}) \otimes \Lambda \to K_0(\cat{MHS}) \otimes\Lambda^{\otimes 2}\] defined by \[ \Delta \ch_n(M) = \sum_{m = 0}^n \ch_{m,n-m}(\Res^{\bbS_n}_{\bbS_m \times \bbS_{n - m}} M). \]
It has the concrete formula given by extending the assignment \( \Delta: p_j \mapsto p_j^{(1)} + p_j^{(2)}. \) 

We extend this to the multisymmetric setting with $s$-tuples of symmetric functions for an integer $s>0.$ More precisely, the discussion above extends to an isomorphism
\[K_0(\cat{MHS}) \otimes \Lambda^{\otimes s}  \cong \bigoplus_{n_1, \ldots, n_s \geq 0} K_0(\cat{MHS}) \otimes \mathrm{Cl}_{\QQ}(\bbS_{n_1} \times \cdots \times \bbS_{n_s})\]
assembled from the following: let $(\Lambda^{\otimes s})_{n_1, \ldots, n_s}$ be the multidegree $(n_1, \ldots, n_s)$-part of $\Lambda^{\otimes s}$, we have
\( K_0(\cat{MHS}) \otimes \Lambda^{\otimes s}_{n_1,\ldots, n_s} \cong K_0(\cat{MHS}_{\bbS_{n_1} \times \cdots \times \bbS_{n_s} }) \otimes \QQ \) again by \cite[Theorem 3.2]{GetzlerPreprint}, given by the Frobenius characteristic \(\ch_{n_1, \ldots, n_s}(M)  \in K_0(\cat{MHS}) \otimes\Lambda^{\otimes s}_{n_1, \ldots, n_s}.\)

\begin{defn}\label{defn:Deltas}
    For $s>0,$ define \(\Delta^{s}: K_0(\cat{MHS}) \otimes\Lambda \to K_0(\cat{MHS})\otimes\Lambda^{\otimes s}\) by \[ \Delta^s(\ch_n(M)) = \sum_{\substack{{n_1 + \cdots +n_s = n}\\{n_i \geq 0 \forall i}}} \ch_{n_1, \ldots, n_s}(\mathrm{Res}^{\bbS_n}_{\bbS_{n_1} \times \cdots \times \bbS_{n_s}} M) \in \Lambda^{\otimes s}. \] Its variant for bisymmetric functions is \(\Delta_2^{s}: K_0(\cat{MHS}) \otimes\Lambda \otimes \Lambda \to K_0(\cat{MHS}) \otimes \Lambda \otimes \Lambda^{s}\) defined as \begin{equation}\label{eqn:delta_2}
    \Delta_2^s(\ch_{m,n}(M)) = \sum_{\substack{{n_1 + \cdots + n_s = n}\\{n_i \geq 0\,\forall\,i}}} \ch_{m, n_1, \ldots, n_s}(\Res^{\bbS_m \times \bbS_n}_{\bbS_{m} \times \bbS_{n_1} \times \cdots \times \bbS_{n_s}}M ).
\end{equation}

Under the isomorphism \(K_0(\cat{MHS})\otimes\Lambda^{\otimes s}\cong K_0(\cat{MHS})\otimes \QQ[\![p_j^{(1)}, \ldots, p_{j}^{(s)}\mid j>0 ]\!],\) a concrete formula of $\Delta^{s}$ is given by extending $p_j\mapsto \sum_{i = 1}^s p_j^{(i)}.$
\end{defn}

\subsection{Invariants} For an $\bbS_n$-equivariant mixed Hodge structure $M$, let $M_{\bbS_n} \subset M$ denote the subspace of $\bbS_n$-invariants. We can calculate the class $[M_{\bbS_n}] \in K_0(\cat{MHS})$ from $\ch_n(M) \in \Lambda$: if $q$ is an indeterminate, we have
\[ [M_{\bbS_n}] \cdot q^n =  \ch_n(M)|_{p_i \to q^i\,\forall\,i}. \]
While setting $q = 1$ recovers $[M_{\bbS_{n}}]$, the substitution $p_i \to q^i$ will be useful for keeping track of degrees of mapping spaces later on. 

More generally, for an $\bbS_{n_1} \times \cdots \times \bbS_{n_s}$-equivariant mixed Hodge structure $M,$ we have
\[ [M_{\bbS_{n_1} \times \cdots \times \bbS_{n_s}}] \cdot q^{\sum_{j} n_j} = \ch_{n_1, \ldots, n_s}(M)|_{p_i^{(j)} \mapsto q^i \,\forall\,i} \in K_0(\cat{MHS)} \cdot q^{\sum_j k_j}. \]
The invariants of $M$ under the subgroup
\(\bbS_{n_2} \times \cdots \times \bbS_{n_s} \subset \bbS_{n_1} \times \cdots \times \bbS_{n_s} \) is an $\bbS_{n_1}$-equivariant mixed Hodge structure, which will be useful for us later on. Its character is computed by the formula
\begin{equation}\label{eqn:multisymmetric_invariants}
    [M_{\bbS_{n_2} \times \cdots \times \bbS_{n_s}}] \cdot q^{\sum_{j} n_j} = \ch_{n_1, \ldots, n_s}(M)|_{\substack{{p_i^{(1)} \mapsto p_i\,\forall\, i>0}\\{p_i^{(j)} \mapsto q^i \,\forall\,i>0,\,j>1}}} \in K_0(\cat{MHS};\bbS_{n_1}) \cdot q^{\sum_j n_j}
\end{equation}
\subsection{Plethysm and graded $\bbS$-spaces} \label{subsec:plethysmgraded} Consider the power series ring $K_0(\cat{MHS}) \otimes \Lambda[\![q]\!]$: it is bigraded, and elements of bidegree $(n, d)$ are precisely those of the form $f\cdot q^d$ where \[f \in K_0(\cat{MHS}) \otimes\Lambda = K_0(\cat{MHS}) \otimes \QQ[\![p_1, p_2, \ldots]\!]\]
is homogeneous of degree $n$ in the $p_i$'s (here, recall that $p_i$ has degree $i$). Let
\[ (K_0(\cat{MHS)} \otimes\Lambda[\![q]\!])^* \subset K_0(\cat{MHS) \otimes}\Lambda[\![q]\!]  \]
denote the subspace with no bidegree $(0,0)$-terms.
\begin{defn}\label{def:Adams}
For each $n > 0$, let \( \psi_n : K_0(\cat{MHS}) \otimes \Lambda[\![q]\!] \to  K_0(\cat{MHS}) \otimes\Lambda[\![q]\!] \) denote the $\QQ$-algebra homomorphism determined as follows:
\begin{enumerate}
    \item $\psi_n(p_m) = p_{mn}$ for all $m$;
    \item $\psi_n(q) = q^n$;
    \item for any $[M] \in K_0(\cat{MHS})$, the element $\psi_n([M]) \in K_0(\cat{MHS})$ is determined recursively by $\psi_1([M]) = [M]$ and the formula
    \[ n[\Sym^n(M)] = \psi_n([M]) + \psi_{n-1}([M]) [\Sym^{1}(M)] + \cdots + \psi_1([M]) [\Sym^{n-1}(M) ].\]
\end{enumerate}
\end{defn}
The homomorphisms $\psi_n$ are the Adams operations for the natural $\lambda$-ring structure on $ K_0(\cat{MHS}) \otimes\Lambda[\![q ]\!]$. With the $\psi_n$ in hand, we can now give the algebraic definition of plethysm.
\begin{defn}\label{defn:plethysmgraded}
    Plethysm, denoted $\circ$, is the associative operation
    \[   K_0(\cat{MHS}) \otimes\Lambda[\![q]\!] \times  ( K_0(\cat{MHS}) \otimes\Lambda[\![q]\!])^* \to  \Lambda[\![q]\!]  \]
    characterized by the following properties:
    \begin{enumerate}
        \item For any $g \in  (K_0(\cat{MHS}) \otimes\Lambda[\![q]\!])^*$, the map $ K_0(\cat{MHS}) \otimes\Lambda[\![q]\!] \to  K_0(\cat{MHS}) \otimes\Lambda[\![q]\!]$ by $f \mapsto f \circ g$ is a $\QQ$-algebra homomorphism;
        \item For any $g \in ( K_0(\cat{MHS}) \otimes\Lambda[\![q]\!])^*$ and $n > 0$, we have $p_n \circ g = \psi_n(g)$.
    \end{enumerate}
\end{defn}

\begin{rem}\label{remark:plethysm}
    Note that the plethysm $f \circ g$ makes sense even if $g \in K_0(\cat{MHS}) \otimes \Lambda[\![q]\!]$ has bidegree $(0,0)$, as long as the bidegree-$(n, d)$ part of $f$ vanishes for $n \gg 0$.
\end{rem}
Plethysm can be interpreted in terms of composition of polynomial functors via Schur--Weyl duality. It has a concrete combinatorial interpretation in terms of \textit{$\bbS$-spaces}, as developed by Getzler--Pandharipande \cite{GetzlerPandharipande}.
\begin{defn}
    An $\bbS$-space $\calX$ is an $\bbS_n$-variety $\calX(n)$ for each $n\geq 0$. A \textit{graded} $\bbS$-space $\calX$ is an $\bbS_n$-variety $\calX(n, d)$ for all $n,d \geq 0$.
\end{defn}

For ease of notation, we consider $\bbS$-spaces as graded $\bbS$-spaces concentrated in degree $d = 0$. The \textit{Serre characteristic} of a graded $\bbS$-space $\calX$ is defined by
\[ \cat{e}^{\bbS}(\calX) := \sum_{n, d} \cat{e}^{\bbS_n}(\calX(n,d)) q^d \]
where $\cat{e}^{\bbS_n}(\calX(n, d))$ denotes the $\bbS_n$-equivariant Serre characteristic of the $\bbS_n$-variety $\calX(n, d)$.

The box product of graded $\bbS$-spaces $\calX$ and $\calY$ is defined by 
\[ (\calX \boxtimes \calY)(n, d) := \coprod_{e = 0}^d \coprod_{m = 0}^n \Ind_{\bbS_m \times \bbS_{n - m}}^{\bbS_{n}} \calX(m, e) \times \calY(n - m, d-e). \]
If $\calX$ and $\calY$ are graded $\bbS$-spaces with either $\calX(n, d) = \varnothing$ for $n\gg 0$ or $\calY(0,0) = 0$, then we can define a graded $\bbS$-space $\calX \circ \calY$ by
\begin{equation}\label{eqn:comp}
(\calX \circ \calY)(n, d) = \bigoplus_{e = 0}^d\bigoplus_{k > 0} \frac{\calX(k, e) \times \calY^{\boxtimes k}(n, d-e)}{\bbS_k}.
\end{equation}

Let $K_0(\cat{Var},\bbS)$ be the Grothendieck group of $\bbS$-spaces, then the box product $\boxtimes$ makes $K_0(\cat{Var},\bbS)$ into a ring. The Grothendieck ring of \textit{graded} $\bbS$-spaces is in turn $K_0(\cat{Var},\bbS)[\![q]\!]$. The Serre characteristic defines a ring homomorphism
\[ \cat{e}^{\bbS}: K_0(\cat{Var},\bbS)[\![q]\!] \to K_0(\cat{MHS}) \otimes \Lambda[\![q]\!] \]
which takes composition of $\bbS$-spaces to plethysm in $\Lambda[\![q]\!]$, respectively:
\[ \cat{e}^{\bbS}(\calX \circ \calY) = \cat{e}^{\bbS}(\calX) \circ \cat{e}^{\bbS}(\calY). \]
See \cite[\S 5]{GetzlerPandharipande} for more details on the Serre characteristic. 
\subsection{Bisymmetric plethysm and weighted stable maps} For $\ell > 0$, a  \textit{graded $\bbS^\ell$-space} is a sequence of varieties $\calX(n_1, \ldots, n_\ell, d)$ for each tuple $(n_1, \ldots, n_\ell, d)$ of nonnegative integers, such that $\calX(n_1, \ldots, n_\ell, d)$ carries an action of $\bbS_{n_1} \times \cdots \times \bbS_{n_\ell}$. Relevant to us will be the cases $\ell = 1,2.$

We introduce the graded $\bbS^\ell$-spaces of interest and give plethystic formulas that compare their equivariant Serre characteristics. The techniques follow previous work on the equivariant E-polynomials of Hassett spaces \cite{hodgehl}.

\begin{defn}\label{defn:catMbar}
    We define graded $\bbS$-spaces
    \[\Mbar_{g,r,N}: (n,d)\mapsto  \Mbar_{g, n}(\Gr(r, N), d),\]
    \[\Mbar_{g, r, N}^{nrt}: (n,d)\mapsto \Mbar_{g, n}^{nrt}(\Gr(r, N), d).\] For each $\ve>0,$ we define graded $\bbS^{2}$-spaces \[\Mbar_{g, r, N}^{\ve}: (m, n, d)\mapsto \Mbar_{g, m|\ve^n}(\Gr(r, N), d),\] \[\Mbar_{g, r, N}^{\ve,nrt}: (m, n, d) \mapsto \Mbar_{g, m|\ve^n}^{nrt}(\Gr(r, N), d),\] \[\Mbar_{g, r, N}^{0}: (m, n, d)\mapsto \Mbar_{g, m|\ve^n}^{nrt}(\Gr(r, N), d) \text{ for any }\ve<1/n.\]
    Note that the choice of $\ve$ in defining $\Mbar_{g, r, N}^{0}$ depends on the $\bbS$-grading $n.$
    Taking the Serre characteristic recovers
    \(\cat{e}^{\bbS}(\Mbar_{g, r, N}) = \overline{\cat{M}}_{g, r, N}\) from the introduction.
    We also define \(\overline{\cat{M}}^{nrt}_{g, r, N} : = \cat{e}^{\bbS}(\Mbar_{g, r, N}^{nrt}). \)
\end{defn}

The following composition operations between graded $\bbS^2$- and $\bbS$-spaces extend the definition from \cite{hodgehl} in the ungraded case.

\begin{defn}
    Let $\calA$ be a graded $\bbS \times \bbS$-space and $\calB$ a graded $\bbS$-space, we define \[(\calA \circ_2 \calB)(m, n, d) = \coprod_{e = 0}^d \coprod_{k > 0} \frac{\calX(m, k, e) \times \calB^{\boxtimes k}(n, d-e)}{\bbS_k}. \]

    Upon taking Serre characteristics, the operations descend to plethysm operations \[\circ_1, \circ_2: K_0(\cat{MHS) \otimes}\Lambda^{\otimes2}[\![ q]\!] \times (K_0(\cat{MHS}) \otimes \Lambda[\![ q ]\!])^* \to K_0(\cat{MHS}) \otimes \Lambda^{\otimes 2}[\![q]\!],\] such that $\cat{e}^{\bbS \times \bbS}(\calA \circ_i \calB) = \cat{e}^{\bbS \times \bbS}(\calA) \circ_i \cat{e}^{\bbS}(\calB).$ 
    
    The plethysm $\circ_i$ is determined by prescriptions \( p_n^{(j)} \circ_i f = p_n^{(j)}\quad\mbox{and}\quad p_n^{(i)} \circ_i f = \psi_n(f)^{(i)}\) and that $f \mapsto f \circ_i g$ defines a $\QQ$-algebra homomorphism for fixed $g \in  (K_0(\cat{MHS}) \otimes\Lambda[\![q]\!])^*,$ where $\{i,j\} = \{1,2\}.$ 
\end{defn}

\begin{defn}
    Let $\sigma_k$ be the graded $\bbS$-space supported in degree $(k, 0)$, where it is $\Spec \C$ with trivial $\bbS_k$-action. Let $\calT_{j}=\bigsqcup_{k=1}^j \sigma_k$ and $\calT_{\infty} = \bigsqcup_{k=1}^\infty \sigma_k.$
    
    We have $\mathsf{e}^{\bbS}(\sigma_k) = h_k,$ the $\bbS_k$-character of the trivial representation. For $\sigma\in \bbS_k,$ let $\lambda_i(\sigma)$ denote its number of $i$-cycles, then Frobenius characteristic formula gives \[h_k = \frac{1}{k!} \sum_{\sigma \in \bbS_k} \prod_{i > 0} p_i^{\lambda_i(\sigma)}.\]
\end{defn}

We recall the plethystic exponential and its truncations for later use.

\begin{defn}
    For any integer $\ell > 0$ and $f \in (K_0(\cat{MHS}) \otimes \Lambda [\![q]\!])^*$, we define
    \[\Exp_{\leq \ell}(f) = \sum_{k = 1}^{\ell} h_k \circ f, \text{ and }\Exp (f) = \sum_{k > 0} h_k \circ f.\]\end{defn}
With these definitions, we have \(\cat{e}^{\bbS}(\calT_j) = \Exp_{\leq j}(p_1)\) and \(\cat{e}^{\bbS}(\calT_{\infty}) = \Exp(p_1). \)
For the next lemma, we will need an enrichment of the coproduct $\Delta$ on $K_0(\cat{MHS}) \otimes \Lambda[\![q]\!]$ to an operation on $\bbS$-spaces. 
\begin{defn}
For a graded $\bbS$-space $\calX$, set \( \Delta \calX (m, n, d) := \Res_{\bbS_m \times \bbS_n}^{\bbS_{m + n}} \calX(m +n, d). \)
This defines a map of Grothendieck rings
\( \Delta: K_0(\cat{Var},\bbS)[\![q]\!] \to K_0(\cat{Var}, \bbS \times \bbS)[\![q]\!] ,\)
which specializes to the coproduct $\Delta$ on $K_0(\cat{MHS}) \otimes \Lambda [\![ q]\!]$ (Definition \ref{defn:Deltas}) upon taking Serre characteristics.
\end{defn}
\begin{lem}\label{lem:Grothendieck_ring_collisions}
    Fix $\ve > 0$ and set $M_{\ve} = \lfloor1/\ve \rfloor$. We have
    \[ [\Mbar^{\ve,nrt}_{g, r, N}] = [\Delta \Mbar^{nrt}_{g, r, N}] \circ_2 \calT_{M_{\ve}} \]
    as elements of $K_0(\cat{Var};\bbS \times \bbS)$. We also have
    \[ [\Mbar^{0,nrt}_{g, r, N}] = [\Delta\Mbar^{nrt}_{g, r, N}] \circ_2 \calT_{\infty}. \]
\end{lem}
\begin{proof}
    The proof is a basic generalization of the proof of \cite[Theorem 3.2]{GetzlerPandharipande}: if we let
    \[ \Mbar_{g, n|\ve^k}^{nrt, j}(\Gr(r, N), d) \subset \Mbar_{g, n|\ve^k}^{nrt}(\Gr(r, N), d) \]
    be the locus where there exactly $j$ distinct marked points among the $k$ points of weight $\ve$, we have a stratification
    \[\Mbar_{g, n|\ve^k}^{nrt, j}(\Gr(r, N), d) = \frac{\Res^{\bbS_{n + j}}_{\bbS_{n} \times \bbS_j} \Mbar_{g, n+ j}^{nrt}(\Gr(r, N), d) \times\coprod_{\substack{k_1 + \cdots + k_j = k\\0 < k_i \leq M_{\ve} \,\forall\,i}}  \Ind_{\bbS_{k_1} \times \cdots \times \bbS_{k_j}}^{\bbS_k} \Spec(\C)}{\bbS_j} \]
    since \[ \coprod_{\substack{k_1 + \cdots + k_j = k\\0 < k_i \leq M_{\ve}}} \Ind_{\bbS_{k_1} \times \cdots \times \bbS_{k_j}}^{\bbS_k} \Spec(\C)\] is the set of ordered partitions of a $\{1,\ldots,k\}$ with $j$ parts, such that each part has size at most $M_{\ve}$. Upon summing over $j$, the lemma follows from the definition of $\Delta$ and of the composition $\circ_2$.
\end{proof}
The following corollary is obtained by applying the Serre characteristic to Lemma \ref{lem:Grothendieck_ring_collisions}. 
\begin{cor}\label{cor:serre_char_collisions}
    For any $\ve > 0$, we have
    \[ \overline{\cat{M}}_{g,r, N}^{nrt, \ve} = \Delta \overline{\cat{M}}_{g,r, N}^{nrt} \circ_2 \Exp_{\leq M_{\ve}}(p_1). \]
    In the $\ve \to 0^+$ limit, we have
     \[ \overline{\cat{M}}_{g,r, N}^{nrt, 0} = \Delta \overline{\cat{M}}_{g,r, N}^{nrt} \circ_2 \Exp(p_1). \]
\end{cor}
\subsection{Symmetric powers}
We recall the following description of the Serre characteristic of symmetric power in terms of plethysm. This will be useful in when we work with relative powers of curves in the next section.
\begin{lem}\label{lem:symmetric_powers}
    For any variety $X$, we have \( h_k \circ \cat{e}(X) = \cat{e}(\Sym^k(X)). \)
\end{lem}
\begin{proof}
   Considering $X$ as a graded $\bbS$-space concentrated in degree $(0, 0)$, unpacking the definition of the composition operation (\ref{eqn:comp}) leads to \((\sigma_k \circ X)(0,0)] = [\Sym^k(X)] \in K_0(\cat{Var}). \)
  The lemma follows by taking $\cat{e}(\cdot)$ on both sides.
\end{proof}
For later use, we compute the Serre characteristics of symmetric powers of $\P^{r - 1}$. While the formula should be standard, we include a proof for completeness.
\begin{lem}\label{lem:sym_of_Pr}
    We have
    \[ \sum_{k = 0}^{\infty} \cat{e}(\Sym^{k}(\P^{r-1}))q^k = \prod_{i = 0}^{r - 1}\frac{1}{1 - \bbL^{i}q}. \]
\end{lem}
\begin{proof}
For a variety $X,$ we specialize the motivic zeta function to Serre characteristics by defining 
\[ Z_{X}(q):= \sum_{k = 0}^{\infty} \cat{e}(\Sym^{k}(X))q^k.\]
It is well-known that $Z_{(\-)}(q)$ is multiplicative with respect to locally closed stratifications, so \[ Z_{\P^{r - 1}}(q) = \prod_{i = 0}^{r - 1} Z_{\A^i}(q). \]
By Lemma \ref{lem:symmetric_powers} and the definition of plethysm, \(\cat{e}(\Sym^{k}(\bbA^i)) = h_k \circ \cat{e}(\bbA^i) = h_k \circ \bbL^i = \bbL^{ik},  \) for any $k \geq 0$. The proof is complete by
\[Z_{\P^{r - 1}}(q) = \prod_{i = 0}^{r - 1} Z_{\A^i}(q) = \prod_{i = 0}^{r - 1}  \left(\sum_{k = 0}^{\infty} \bbL^{ik}q^k\right)  =  \prod_{i = 0}^{r - 1} \frac{1}{1 - \bbL^iq}.\]
\end{proof}

\begin{rem}
    It is simpler to calculate the Serre characteristics $\cat{e}(\Sym^k(\bbA^i))$ than the actual motivic zeta function of $\bbA^i$ (\cite[Lemma 4.4]{Gottsche90}, see also \cite[Proposition 7.32]{Mustata}). 
    
\end{rem}

For an indeterminate $t$ and integers $N, M$, recall the notation 
\[ \qbinomopt[t]{N}{M} = \frac{(1 - t^N)(1 - t^{N-1}) \cdots (1 - t^{N - M + 1})}{(1 - t^M) (1  - t^{M - 1}) \cdots (1- t)}  \]
for the evaluation of the $q$-binomial coefficient (\ref{eqn:q_binom_defn}) at $t$. 
\begin{cor}\label{cor:sym_L_binomial}
    For any $r \geq 1$ and $k \geq 0$, we have
    \[ \cat{e}(\Sym^{k}(\P^{r-1})) =\qbinomopt[\bbL]{r + k - 1}{k} = \cat{e}(\Gr(k, r+ k - 1)) \]
\end{cor}
\begin{proof}
The first equality follows from the well-known generating function identity
\[ \prod_{i = 0}^{r - 1} \frac{1}{1- \bbL^{i}q} = \sum_{k = 0}^{\infty}\qbinomopt[\bbL]{r + k - 1}{k} q^k, \]
see e.g. \cite[p. 68]{EC1}, while the second one is the well-known formula for the Serre characteristic of the Grassmannian, obtained by counting Schubert cells.
\end{proof}

\section{Motives of Quot schemes of relative curves}\label{sec:relative_quot}
The goal of this section is to prove Corollary \ref{cor:nrtqvsm} which gives a comparison formula between the Grothendieck ring classes of $\ve$-stable quasimaps and stable maps to Grassmannian. We begin by computing the class of relative Quot scheme of a family of curves (Proposition \ref{prop:relquot}), which may be of independent interest. We apply this to the moduli spaces of $\ve$-stable quasimaps after stratifying them in terms of such relative Quot schemes.

Let $C\to S$ be a flat family of curves, and let $\calE$ be a vector bundle on $C$.
\begin{defn}
     The relative zero-dimensional degree-$k$ Quot scheme of $\calE,$ denoted as \[\Quot^k_{C/S}(\calE)\to S,\] represents the functor $\mathsf{Sch}_S^{\mathrm{op}}\to \mathsf{Set}$ given by sending a morphism $f: T\to S$ to the set of quotient sheaves \[\calE_T\to \mathcal{Q}\to 0\] on the fiber product $C_T:=C\times_S T$ such that $\calQ$ is flat over $T,$ and for each geometric point $t\in T,$ the restriction to the fiber  $\calE_{C_t}\to \mathcal{Q}|_{C_t}\to 0$ is a zero-dimensional degree-$k$ quotient on the curve $C_t.$
\end{defn}

Concretely, the fiber of $\Quot^k_{C/S}(\calE)\to S$ over $s\in S$ is the zero-dimensional degree-$k$ Quot scheme of the vector bundle $\calE_s$ on $C_s.$ We will also need an $\ve$-weighted version of the relative Quot scheme.

\begin{defn}\label{defn:quotve}
     For $\ve > 0$, denote by 
     \[\Quot^{k, \ve}_{C/S}(\calE) \subset \Quot^k_{C/S}(\calE)\]
     the subscheme which represents the subfunctor sending $f: T \to S$ to the set of quotient sheaves
     \[ \calE_T \to \calQ \to 0 \]
     as above, such that for any point $t \in T$, and $p \in C_t$, the length $\ell(p)$ of the torsion sheaf $\calQ|_{C_t}$ at $p$ satisfies $\ell(p) \leq 1/\ve$.
\end{defn}

When $\ve$ is sufficiently small, we have $\Quot^{k}_{C/S}(\calE) = \Quot^{k, \ve}_{C/S}(\calE)$.

When $S = \mathrm{Spec}(\C),$ and $C$ is smooth, the Grothendieck ring class of $\Quot^k_C(\calE)$ has been calculated in previous works. Abusing notation, we will use the notation $\bbL$ for the class of the affine line in the Grothendieck ring $K_0(\cat{Var})$ of varieties over $\C$; it should be clear from the context whether $\bbL$ is meant as an element of $K_0(\cat{Var})$ or as an element of $K_0(\cat{MHS})$. 

\begin{thm} \label{thm:quotmotive}Let $C\to \mathrm{Spec}(\C)$ be a smooth curve and let $\calE$ be a rank-$r$ vector bundle on $C$.
    \begin{enumerate}
        \item \cite[Theorem 4.1]{BFP20}, \cite[Corollary 2.4]{Ric20}, \cite[Proposition 4.1]{MonavariRicolfi} In $K_0(\mathsf{Var}, \C),$ we have \[[\Quot^k_{C}(\calE)] = [\Quot^k_{C}(\mathcal{O}_C^{r})].\]
        \item \cite{Bif89}, \cite[Proposition 4.5]{BFP20} The Grothendieck ring class of $[\Quot^k_{C}(\calE)] = [\Quot^k_{C}(\mathcal{O}_C^{r})]$ is given by \[\sum_{\substack{k_1 + \cdots +k_r = k\\ k_i\geq 0}} \prod_{i = 1}^{r} \bbL^{(i - 1)k_i} [\Sym^{k_i}(C)]. \]
    \end{enumerate}
\end{thm}

We prove an analogous formula for the motive of Quot schemes of \emph{relative} curves, in the $\ve$-weighted case. The first step is to show that the motive is independent of the vector bundle $\calE$

\begin{lem}\label{lem:indep-Quot}
    Let $\pi: C\to S$ be a flat family of curves, and let $\calE$ be a rank-$r$ vector bundle on $C.$ Then we have $[\Quot^{k, \ve}_{C/S}(\calE)] = [\Quot^{k, \ve}_{C/S}(\calO_C^r)]$ in $K_0(\mathsf{Var}, \C).$
\end{lem}

    \begin{proof}
    The proof is modeled on \cite[Proposition 4.1]{MonavariRicolfi}.

   We observe that the class of $[\mathrm{Quot}^{k, \ve}_{X/S}(\calE)]$ is cut-and-paste with respect to the base $S,$ in the following sense. Let $S = \bigsqcup_{i\in I}S_i$ be a finite locally closed stratification. For each $i$, we have a fiber diagram \[\begin{tikzcd}
	{\mathrm{Quot}^{k, \ve}_{\pi^{-1}(S_i)/S_i}(\calE|_{\pi^{-1}(S_i)})} & {\mathrm{Quot}^{k, \ve}_{X/S}(\calE)} \\
	{S_i} & S
	\arrow[from=1-1, to=1-2]
	\arrow[from=1-1, to=2-1]
	\arrow[from=1-2, to=2-2]
	\arrow[from=2-1, to=2-2]
\end{tikzcd}\]
    
    Therefore, there is an induced locally closed stratification \[\mathrm{Quot}^{k, \ve}_{X/S}(\calE) = \bigsqcup_{i\in I} \mathrm{Quot}^{k, \ve}_{\pi^{-1}(S_i)/S_i}(\calE|_{\pi^{-1}(S_i)}),\] which gives an equality in $K_0(\mathsf{Var}, \C).$ 

    Now, we claim that there exist finite locally closed stratifications \[S = \bigsqcup_{i \in I} S_i \text{ and }\pi^{-1}(S_i) = \bigsqcup_{j \in \calA_i} W_j\] 
    such that $W_j \to S_i$ is flat for all $j \in \calA_i$ and $\calE$ is trivial on each $W_j$ (here $\calA_i$ is some finite set of indices). They can be constructed by choosing an open cover of $X$ which trivializes $\calE$, taking the image open cover of $S$, and suitably refining it to obtain the $S_i$.

    Since we can always find such stratifications, and the class of $\Quot_{X/S}^{k, \ve}(\calE)$ is cut-and-paste with respect to $S$, we can prove the lemma assuming that there is only one $S_i.$ In other words, from now on we assume that
    \[ X = \bigsqcup_{j = 1}^{\ell} W_j \]
    such that $\calE$ is trivial on each $W_j$, and $W_j \to S$ is flat for all $j$.  
    
Consider the morphism \[\Quot^{k, \ve}_{X/S}(\calE)\to \Sym_{X/S}^k,\] which records the support of the zero-dimensional quotient sheaf, and let \[\Quot^{k, \ve}_{W_i \subset X/S}(\calE)\subset \Quot^{k, \ve}_{X/S}(\calE)\] be the pre-image of \(\Sym_{W_i/S}^k \subset \Sym_{X/S}^k.\) We stratify $\Sym_{X/S}^k$ by fiber products \[\bigsqcup_{(k_i\geq 0): \sum_{ i = 1}^\ell k_i = k} \ \Sym_{W_1/S}^{k_1} \times_S \cdots \times_S \Sym_{W_\ell/S}^{k_\ell}, \]
and define
\[ \Quot^{k_1,\ldots,k_\ell;\ve}_{X/S}(\calE) := \pi^{-1}(\Sym_{W_1/S}^{k_1} \times_S \cdots \times_S \Sym_{W_\ell/S}^{k_\ell}). \]
There is an isomorphism
\[\Quot^{k_1,\ldots,k_\ell;\ve}_{X/S}(\calE) \cong \Quot^{k_1,\ve}_{W_1 \subset X/S} (\calE) \times_S \cdots \times_S \Quot^{k_\ell, \ve}_{W_\ell \subset X/S}(\calE). \]
given by direct sums. We now prove the lemma by arguing that for each $i$, we have
\begin{equation}\label{eqn:key_iso} 
\Quot_{W_i \subset X/S}^{k_i, \ve}(\calE) \cong \Quot_{W_i \subset X/S}^{k_i, \ve}(\calO_X^{\oplus r}).
\end{equation}
Indeed, by construction there always exists an open set $U_i \subset X$ containing $W_i$ on which $\calE$ is trivial. Via extension by zero, we have an isomorphism
\[ \Quot_{W_i \subset U_i/S}^{k_i, \ve}(\calE|_{U_i}) \cong \Quot^{k_i, \ve}_{W_i \subset X/S}(\calE) \]
Since $\calE|_{U_i} = \calO^{\oplus r}|_{U_i}$, we obtain the isomorphism (\ref{eqn:key_iso}), and the proof is complete.
\end{proof}
Specializing to a family of smooth curves, we now determine the Grothendieck ring class of the relative Quot scheme in terms of certain relative weighted configuration spaces. 

\begin{defn}
Let
\( C_S^k := \overbrace{(C \times_S \cdots \times_S C)}^{k}\) be the $k$-th fiber power of $C\to S,$
and let $\Delta_d \subset C_S^k$ be the union of all of the closed diagonals defined by the equality of $d$ coordinates. Define
\[ \mathrm{Conf}^k_{\ve}(C/S) := C_S^{k} \smallsetminus \Delta_{\lfloor1/\ve \rfloor + 1}. \]
Over a point $s \in S$, the space $\mathrm{Conf}^k_{\ve}(C/S)$ parameterizes $k$ ordered points in the fiber $C_s$, such that no more than $\lfloor 1/\ve \rfloor$ of the points can coincide.
\end{defn}

The proof of the following formula is reminiscent of T. Mochizuki's work on complete filt schemes \cite{Mochizuki}.

\begin{prop}\label{prop:relquot}
    Suppose $\ve > 0$ and suppose $\overline{C} \to S$ is a flat family of prestable curves, with $n \geq 0$ sections $s_1, \ldots, s_n: S \to \overline{C}$ which land in the smooth locus. Let $C \to S$ be a flat family of smooth curves obtained by deleting the nodes and any subset of the sections from $\overline{C}$. Then \[[\mathrm{Quot}^{k, \ve}_{C/S}(\mathcal{O}_C^r)] = \sum_{\substack{k_1+\cdots + k_r = k\\ k_i\geq 0}}\mathbb{L}^{\sum_{i=1}^{r}(i-1) k_i}\left[\frac{\mathrm{Conf}_{\ve}^k(C/S)}{\bbS_{k_1} \times \bbS_{k_2} \times \cdots \times \bbS_{k_r}}\right]
    \]
    as elements in $K_0(\cat{Var})$.
\end{prop}
\begin{proof}
    We first prove the formula in the case where $\ve$ is very small, i.e. when $\mathrm{Quot}^{k, \ve}_{C/S}(\mathcal{O}_C^r) = \Quot^{k}_{C/S}(\calO^r_C)$ and
\[\frac{\mathrm{Conf}_{\ve}^k(C/S)}{\bbS_{k_1} \times \bbS_{k_2} \times \cdots \times \bbS_{k_r}} = \Sym^{k_1}(C/S) \times_S \cdots \times_S \Sym^{k_r}(C/S).\]
The general case will follow from this one, which we now prove in several steps.
\\
    \noindent \fbox{\textit{Stratification by a flag.}} We fix the standard flag $0\subset \mathcal{F}_1\subset \mathcal{F}_2\subset \cdots\subset \mathcal{F}_r = \mathcal{O}_C^r$ with $\mathcal{F}_i = \bigoplus_{j = 1}^i \mathcal{O}_C.$ Given a point $q = [0\to \mathcal{V}\to \mathcal{O}_{C_s}^r\to \mathcal{T}\to 0],$ we define \[\mathcal{V}_i:=\mathcal{V}\cap \mathcal{F}_i\text{ and }\mathcal{T}_i:=\mathrm{coker}(\mathcal{V}_i\to \mathcal{F}_i),\] so that $\mathcal{T}_i$ are zero-dimensional quotient sheaves of $\mathcal{F}_i.$ Define \[\ell_1(q) :=\mathrm{lt}(\mathcal{T}_1),
    \ell_i(\mathcal{T}) :=  \mathrm{lt}(\mathcal{T}_i) - \mathrm{lt}(\mathcal{T}_{i-1})\text{ for }i>1.
    \] Note that $\mathrm{lt}(\mathcal{T}_r) = \mathrm{lt}(\mathcal{T}) = k.$ The assignment $q\mapsto (\ell_1(q),\dots, \ell_r(q))$ produces a locally closed stratification\footnote{The notation of the strata coincides with the ones defined in the proof of Lemma \ref{lem:indep-Quot}, while the two stratifications are a priori not the same. We apologize for the abuse of notation.} of $\mathrm{Quot}^k_{C/S}(\mathcal{O}_C^r)$ into $\mathrm{Quot}^{\ell_1,\dots, \ell_r}_{C/S}(\mathcal{O}_C^r)$ over $S.$

    \noindent\fbox{\textit{Vector bundles over fiber products of symmetric powers.}} 
    We produce a recursive description of the strata defined above. The base case is when $r = 1,$ in which the stratification is trivial, and we recover the relative symmetric power $\mathrm{Sym}^k(C/S).$ For general $r,$ on the fiber product \[\mathrm{Quot}^{\ell_1,\dots, \ell_{r-1}}_{C/S}(\mathcal{O}_C^{r-1})\times_{S} \mathrm{Sym}^{\ell_r}(C/S)\times_S C,\] let $
    \mathfrak{I}$ be the pullback of the universal ideal sheaf from $\mathrm{Sym}^{\ell_r}(C/S)\times_S C,$ and let \[0\to \mathfrak{V}_{r-1}\to \mathfrak{F}_{r-1}\cong \mathcal{O}_C^{r-1}\] be the pullback of the universal subsheaf on $\mathrm{Quot}^{\ell_1,\dots, \ell_{r-1}}_{C/S}(\mathcal{O}_C^{r-1})\times_{S} C.$ Let $$\pi: \mathrm{Quot}^{\ell_1,\dots, \ell_{r-1}}_{C/S}(\mathcal{O}_C^{r-1})\times_{S} \mathrm{Sym}^{\ell_r}(C/S)\times_S C\to \mathrm{Quot}^{\ell_1,\dots, \ell_{r-1}}_{C/S}(\mathcal{O}_C^{r-1})\times_{S} \mathrm{Sym}^{\ell_r}(C/S)$$ be the projection map. We claim that:
    \begin{enumerate}
        \item $\pi_* \mathscr{H}om(\mathfrak{I}, \mathfrak{F}_{r-1}/\mathfrak{V}_{r-1})$ is a locally free sheaf on $\mathrm{Quot}^{\ell_1,\dots, \ell_{r-1}}_{C/S}(\mathcal{O}_C^{r-1})\times_{S} \mathrm{Sym}^{\ell_r}(C/S)$ of rank $\mathrm{lt}(\mathfrak{F}_{r-1}/\mathfrak{V}_{r-1}) = \sum_{j=1}^{r-1}\ell_j,$
        \item the forgetful map $$\mathrm{Quot}^{\ell_1,\dots, \ell_{r} }_{C/S}(\mathcal{O}_C^{r})\to \mathrm{Quot}^{\ell_1,\dots, \ell_{r-1}}_{C/S}(\mathcal{O}_C^{r-1})\times_{S} \mathrm{Sym}^{\ell_r}(C/S)$$ given by \[[0\to \mathcal{V}\to \mathcal{O}_{C_s}^r\to \mathcal{T}\to 0]\mapsto \left([0\to \mathcal{V}_{r-1}\to \mathcal{O}_{C_s}^{r-1}\to \mathcal{T}_{r-1}\to 0], \mathrm{supp}(\calT/\calT_{r-1})\right),\]
        is isomorphic to the projection map from $\mathrm{Tot}(\pi_* \mathrm{Hom}(\mathfrak{I}, \mathfrak{F}_{r-1}/\mathfrak{V}_{r-1}))$ to the base.
    \end{enumerate}

    To prove (1), we first complete $C\to S$ to $\overline{C}\to S$. Consider now the fiber product
    \[\Quot_{C/S}^{\ell_1,\dots, \ell_{r-1}}(\mathcal{O}_{C}^{r-1})\times_S \Sym^{\ell_r}(C/S)\times_S \overline{C}, \]
    which fits into a diagram
    \[\hspace{-.4cm} \begin{tikzcd}
        &\Quot_{C/S}^{\ell_1,\dots, \ell_{r-1}}(\mathcal{O}_{C}^{r-1})\times_S \Sym^{\ell_r}(C/S)\times_S C \arrow[r, "j"]\arrow[dr, "\pi"] &\Quot_{C/S}^{\ell_1,\dots, \ell_{r-1}}(\mathcal{O}_{C}^{r-1})\times_S \Sym^{\ell_r}(C/S)\times_S \overline{C} \arrow[d, "\overline{\pi}"]\\
        & &\Quot_{C/S}^{\ell_1,\dots, \ell_{r-1}}(\mathcal{O}_{C}^{r-1})\times_S \Sym^{\ell_r}(C/S)
    \end{tikzcd}  \]
    
    where $j$ is an open immersion.  Let $\overline{\mathfrak{I}} := j_* \mathfrak{I}$ and $\overline{\mathfrak{T}}_{r - 1} := j_* (\mathfrak{V}_{r - 1}/\mathfrak{F}_{r - 1})$.
    We claim that there is an equality of sheaves
    \begin{equation}\label{eqn:sheaf_equality} \pi_* \mathscr{H}om(\mathfrak{I}, \mathfrak{F}_{r - 1}/\mathfrak{V}_{r - 1}) = \overline{\pi}_*\mathscr{H}om(\overline{\mathfrak{I}}, \overline{\mathfrak{T}}_{r - 1}). \end{equation}
    To see this, note that
    \[
        \pi_* \mathscr{H}om(\mathfrak{I}, \mathfrak{F}_{r - 1}/\mathfrak{V}_{r - 1}) = \overline{\pi}_* j_*\mathscr{H}om(\mathfrak{I}, \mathfrak{F}_{r - 1}/\mathfrak{V}_{r - 1}),
    \]
    and 
    \[j_*\mathscr{H}om(\mathfrak{I}, \mathfrak{F}_{r - 1}/\mathfrak{V}_{r - 1}) = \mathscr{H}om(\overline{\mathfrak{I}},\overline{\mathfrak{T}}_{r - 1}), \]
    as can be verified at the level of stalks. Over a point \[(q, D) = ([0\to \mathcal{V}_{r-1}\to \mathcal{O}_{C_s}^{r-1}\to \mathcal{T}_{r-1}\to 0], D\subset C_s)\in \mathrm{Quot}^{\ell_1,\dots, \ell_{r-1}}_{C/S}(\mathcal{O}_{C}^{r-1})\times_{S} \mathrm{Sym}^{\ell_r}(C/S),\] the pushforward $\overline{\pi}_*\mathscr{H}om(\overline{\mathfrak{I}}, \overline{\mathfrak{T}}_{r - 1})$ has stalk at $(q, D)$ given by $$\mathrm{Hom}(\mathcal{O}_{C_s}(-D), \mathcal{T}_{r-1})= \mathrm{Hom}(\mathcal{O}_{C_s}, \mathcal{T}_{r-1}(D)) = \mathrm{Hom}(\mathcal{O}_{C_s}, \mathcal{T}_{r-1}) = H^0(C_s, \mathcal{T}_{r-1}).$$Since $\overline{\pi}$ is proper and the stalk always has rank $\mathrm{lt}(\mathcal{T}_{r-1}) = \sum_{j=1}^{r-1}\ell_j,$ Grauert's theorem implies that \[\overline{\pi}_*\mathscr{H}om(\overline{\mathfrak{I}}, \overline{\mathfrak{T}}_{r - 1})\] is a locally free sheaf of rank $\Quot_{C/S}^{\ell_1,\dots, \ell_{r-1}}(\mathcal{O}_{C}^{r-1})\times_S \Sym^{\ell_r}(C/S)$. By the equality (\ref{eqn:sheaf_equality}) we see that $\pi_* \mathscr{H}om({\mathfrak{I}}, {\mathfrak{F}}_{r-1}/\mathfrak{V}_{r-1})$ is locally free of rank $\sum_{j=1}^{r-1}\ell_j$, which proves (1).

    For (2), we take the fiber product with $C$ on both sides of the forgetful map to get \[\Quot_{C/S}^{\ell_1,\dots, \ell_{r}}(\mathcal{O}_{C}^{r})\times_S C\to \Quot_{C/S}^{\ell_1,\dots, \ell_{r-1}}(\mathcal{O}_{C}^{r-1})\times_S \Sym^{\ell_r}(C/S)\times_S C.\]

    Abusing notation, we denote the pullback of the universal sheaves $\mathfrak{I}, \mathfrak{V}_{r-1},$ and $\mathfrak{F}_{r-1}$ along the morphism with the same notation. On the other hand, let $0\to \mathfrak{V}'_{r-1}\to \mathfrak{F}'_{r-1}$ and $0\to \mathfrak{V}'_{r}\to \mathfrak{F}'_{r}$ be the universal subsheaves on $\Quot_{C/S}^{\ell_1,\dots, \ell_{r}}(\mathcal{O}_{C}^{r})\times_S C.$ By construction, there are isomorphisms $[0\to \mathfrak{V}'_{r-1}\to \mathfrak{F}'_{r-1}]\cong [0\to \mathfrak{V}_{r-1}\to \mathfrak{F}_{r-1}]$ and $\mathfrak{I}\cong \mathfrak{V}'_{r}/\mathfrak{V}'_{r-1}.$ Therefore, the composition \[\mathfrak{I}\cong \mathfrak{V}'_{r}/\mathfrak{V}'_{r-1}\to \mathfrak{F}'_{r}/\mathfrak{V}'_{r-1} = \frac{\mathfrak{F}'_{r-1}\oplus \mathcal{O}}{\mathfrak{V}'_{r-1}} \cong \mathcal{O}\oplus \mathfrak{F}'_{r-1}/\mathfrak{V}'_{r-1}\to \mathfrak{F}'_{r-1}/\mathfrak{V}'_{r-1}\cong \mathfrak{F}_{r-1}/\mathfrak{V}_{r-1}\] on $\Quot_{C/S}^{\ell_1,\dots, \ell_{r}}(\mathcal{O}_{C}^{r})\times_S C$ produces a morphism\footnote{Let $\pi: Y\to Z,$ then the universal property of $\mathrm{Tot}(\pi_* \mathfrak{F})$ is that: a morphism $X\to \mathrm{Tot}(\pi_* \mathfrak{F})$ is the same as a morphism $X\to Z$ and a section of the pullback of 
    $\mathfrak{F}$ on $X\times_Z Y.$} \[\Quot_{C/S}^{\ell_1,\dots, \ell_{r}}(\mathcal{O}_{C}^{r})\to \mathrm{Tot}(\pi_* \mathscr{H}om(\mathfrak{I}, \mathfrak{F}_{r-1}/\mathfrak{V}_{r-1})).\] In the other direction, letting $\mathfrak{I}, \mathfrak{V}_{r-1},$ and $\mathfrak{F}_{r-1}$ be the pullback of the universal sheaves to $\mathrm{Tot}(\pi_* \mathscr{H}om(\mathfrak{I}, \mathfrak{F}_{r-1}/\mathfrak{V}_{r-1}))\times_S C,$ there is a universal morphism $\mathfrak{I}\to \mathfrak{F}_{r-1}/\mathfrak{V}_{r-1}$ on the fiber product. Let $\phi:\mathfrak{I}\to \mathfrak{F}_{r-1}/\mathfrak{V}_{r-1}\to \mathfrak{F}_{r}/\mathfrak{V}_{r-1}$ be the composition, and let $\mathfrak{V}'_r\subset \mathfrak{F}_r$ be the preimage of $\phi(\mathfrak{I})\subset \mathfrak{F}_{r}/\mathfrak{V}_{r-1}.$ By construction, $\mathfrak{V}'_r\subset \mathfrak{F}_r$ on $\mathrm{Tot}(\pi_* \mathscr{H}om(\mathfrak{I}, \mathfrak{F}_{r-1}/\mathfrak{V}_{r-1}))\times_S C$ defines a morphism \[\mathrm{Tot}(\pi_* \mathscr{H}om(\mathfrak{I}, \mathfrak{F}_{r-1}/\mathfrak{V}_{r-1}))\to \Quot_{C/S}^{\ell_1,\dots,\ell_r}(\mathcal{O}_C^r),\] which is the inverse of the morphism described above, and both isomorphisms commute with the forgetful map to $\mathrm{Quot}^{\ell_1,\dots, \ell_{r-1}}_{C/S}(\mathcal{O}_C^{r-1})\times_{S} \mathrm{Sym}^{\ell_r}(C/S).$

    \noindent\fbox{\textit{Formula.}} Iterating the vector bundle description, we see that \begin{equation}\label{eqn:iter_vb}
        [\Quot_{C/S}^{\ell_1,\dots, \ell_{r}}(\mathcal{O}_{C}^{r})] = \mathbb{L}^{\sum_{m = 1}^{r-1}\sum_{j=1}^{m}\ell_j}[\Sym^{\ell_1}(C/S)\times_S\times\cdots\times_S \Sym^{\ell_r}(C/S)],
    \end{equation} and $\sum_{m=1}^{r-1}\sum_{j=1}^m \ell_j = \sum_{i=1}^{r-1}(r-i)\ell_i.$ Setting $\ell_i = k_{r+1-i}$ leads to the formula in the case of small $\ve$.\\
    \noindent\fbox{\textit{General} $\ve$.} In the case of general $\ve$, we obtain a stratification of $\Quot^{k, \ve}_{C/S}(\calO^r_C)$ by the strata
    \[\Quot^{\ell_1, \ldots, \ell_r;\ve}_{C/S}(\calO_C^r) = \Quot^{\ell_1, \ldots, \ell_r}_{C/S}(\calO_C^r) \cap \Quot^{k, \ve}_{C/S}(\calO_C^r ).\]
    For each $\ell_1, \ldots, \ell_r$, we obtain the following fiber diagram:
    \[
    \begin{tikzcd}
        &\Quot^{\ell_1, \ldots, \ell_r;\ve}_{C/S}(\calO_C^r) \arrow[r] \arrow[d]&\Quot^{\ell_1, \ldots, \ell_r}_{C/S}(\calO_C^r)\arrow[d] \\
        &\frac{\mathrm{Conf}_{\ve}^k(C/S)}{\bbS_{\ell_1} \times \bbS_{\ell_2} \times \cdots \times \bbS_{\ell_r}}\arrow[r] &\Sym^{\ell_1}(C/S)\times_S\times\cdots\times_S \Sym^{\ell_r}(C/S)
    \end{tikzcd}.
    \]   
    Since the right vertical arrow factors as a sequence of vector bundles, the left vertical arrow must also factor as a sequence of vector bundles of the same ranks. Therefore (\ref{eqn:iter_vb}) holds for general $\ve$, if we replace the relative product of relative symmetric powers by the corresponding quotient of the $\ve$-weighted configuration space.
\end{proof}

\subsection{Stratifying $\ve$-stable quasimaps over smooth curves}
In this section, we relate the torsion length stratification of $\Q^{\ve}_{g,n}(\Gr(r,N),d)$ to weighted maps from smooth curves: see Definitions \ref{defn:vequasimaps}, \ref{defn:torsionstrat}, and \ref{defn:wtstablemaps} for the relevant notation.

\begin{notn}
    Let
\( \calC_{g, n}(\Gr(r, N), d-k) \to \M_{g, n}(\Gr(r, N), d-k) \) be the universal curve, and use \( \calC_{g, n}^\circ(\Gr(r, N), d-k) \subset \calC_{g, n}(\Gr(r, N), d-k)\)
to denote the complement of the $n$ marking sections. 
 
Consider the evaluation map \(ev: \calC_{g, n}^{\circ}(\Gr(r, N), d-k) \to \Gr(r, N). \) With $\calS$ denoting the tautological rank-$r$ vector bundle on $\Gr(r, N),$ we obtain a vector bundle $ev^*\calS$ on $\calC_{g, n}^\circ(\Gr(r, N), d-k)$.
\end{notn}

\begin{lem}\label{lem:smooth_qstrata}
There is an isomorphism 
    \[ \Q_{g, n}^{\ve,k}(\Gr(r, N), d) \cong \Quot^k_{\calC^{\circ}_{g, n}(\Gr(r, N), d-k)/\M_{g, n}(\Gr(r, N), d-k)}(ev^*\calS). \]
\end{lem}
\begin{proof}
Let $(C, q_1, \ldots, q_n)$ be a smooth $n$-pointed curve and let \( 0 \to \calV \to \calO_C^{\oplus N} \to \calW \to 0 \)
be a $\ve$-stable quotient of degree $d$ over $C$, which determines a point of $\Q_{g, n}^{\ve,k}(\Gr(r, N), d)$. In particular, $\calV$ is locally free of rank $r$, and the torsion summand of the quotient sheaf
\( \calW^{tors} \subset \calW  \) has length $k.$
 
 By composing with the projection $\calW \to \calW^{free}:= \calW/\calW^{tors},$ we obtain a locally free quotient \[0 \to \calV^{sat} \to \calO_C^{\oplus N}\to \calW^{free} \to 0,\] where $\calV^{sat}$ is locally free of rank $r$. By the universal property of $\Gr(r, N)$, this determines a degree-$(d -k)$ morphism from $C$ to $\Gr(r, N)$, such that $\calV^{sat}$ is the pullback of the tautological bundle $\calS$. Note also that there is a short exact sequence
\[ 0 \to \calV \to \calV^{sat} \to \calW^{tors} \to 0,  \]
which determines the relative degree-$k$ quotient of $ev^*\calS$. We have thus constructed a morphism
\begin{equation}\label{eqn:Qtoquot_morphism}\Q_{g, n}^{\ve,k}(\Gr(r, N), d) \to \Quot^k_{\calC^{\circ}_{g, n}(\Gr(r, N), d-k)/\M_{g, n}(\Gr(r, N), d-k)}(ev^*\calS).
\end{equation}

An inverse can be constructed as follows. Let $f: C \to \Gr(r, N)$ be a degree-$(d-k)$ and let \( 0 \to \calE \to  f^*\calS \to \calT \to 0 \) and a degree-$k$ torsion quotient on $C.$ The map $f$ induces a short exact sequence
\(0 \to f^*\calS \to \calO_C^{\oplus N} \to \calW \to 0.\)

 We compose the inclusions $\calE \to f^*\calS$ and $f^*\calS \to \calO_C^{\oplus N}$ 
to get a short exact sequence
\[ 0 \to \calE \to \calO^{\oplus N}_C \to \calW' \to 0 \]
where $(\calW')^{tors} = \calT$ and $\calW'/\calT = \calW$. Since we have subtracted the sections from $\calC_{g, n}^{nrt}(\Gr(r, N), d-k)$, this procedure produces a morphism \[ \Quot^{\ve,k}_{\calC^{\circ}_{g, n}(\Gr(r, N), d-k)/\M_{g, n}(\Gr(r, N), d-k)}(ev^*\calS) \to \Qbar_{g, n}^{\ve, k}(\Gr(r, N), d)\]
and it is an inverse to (\ref{eqn:Qtoquot_morphism}). This completes the proof.
\end{proof}
Now, since
\[ \M_{g, n|\ve^k}(\Gr(r, N), d-k) = \mathrm{Conf}_{\ve}^{k}(\calC_{g, n}^\circ(\Gr(r, N), d-k)/\M_{g, n}(\Gr(r, N), d-k)), \]
we obtain the following corollary.

\begin{cor}\label{cor:Qstrat_over_smooth}
    We have the following equality in $K_0(\cat{Var})$:
    \[[\Q_{g, n}^{\ve,k}(\Gr(r, N), d)] = \sum_{\substack{k_1+\cdots + k_r = k\\ k_i\geq 0}}\mathbb{L}^{\sum_{i=1}^{r}(i-1) k_i}\left[\frac{\M_{g, n|\ve^{k}}(\Gr(r, N), d-k)}{\bbS_{k_1} \times \bbS_{k_2} \times \cdots \times \bbS_{k_r}}\right].\]
    The same equality holds for $\bbS_n$-equivariant Serre characteristics.
\end{cor}
\begin{proof}
    The statement about the equality in the Grothendieck ring is an immediate corollary of Lemma \ref{lem:smooth_qstrata} and Proposition \ref{prop:relquot}. The statement about $\bbS_n$-equivariant Serre characteristics follows the same proof as Proposition \ref{prop:relquot}, using the following fact: if $\calE \to X$ is an $\bbS_n$-equivariant rank-$s$ vector bundle, then
    \(\cat{e}^{\bbS_n}(\calE) = \bbL^{s} \cat{e}^{\bbS_n}(X). \)
    
    This is proved using the Thom isomorphism
    \( H^{k}_c(\calE;\QQ) \cong H_c^{k}(X;\QQ)(-r), \)
    induced by cupping by the Thom class of $\calE$. Since $\bbS_n$ acts by algebraic automorphisms, this Thom class is preserved by $\bbS_n,$ so the Thom isomorphism above is $\bbS_n$-equivariant.
\end{proof}

\subsection{Stratifying $\ve$-stable quasimaps over nodal curves}
The goal of this section is to prove Theorem \ref{thm:relquot-to-weightedstablemaps}, a comparison formula between a stratification of the moduli space of $\ve$-stable quasimaps with no rational tails and the heavy-light stable maps. It extends Corollary \ref{cor:Qstrat_over_smooth} to maps from maps without rational tails and is the key geometric ingredient of the main formula in Theorem \ref{thm:main}.

We recall that
\(\Mbar_{g, n}^{nrt}(\Gr(r, N),d)\) resp. \(\Qbar_{g, n}^{\ve,nrt}(\Gr(r, N),d)\)
denotes the locus of stable maps resp. $\ve$-stable quasimaps whose source curves do not have any rational tails and that, for $0\leq k\leq d,$ \(\Qbar_{g, n}^{\ve, nrt, k}(\Gr(r, N), d)\) denotes the locus of $\ve$-stable quasimaps over curves without rational tails, such that the torsion of the quasimap has length $k.$ We now work towards the following analogue of Corollary \ref{cor:Qstrat_over_smooth}, by generalizing the proof of Proposition \ref{prop:relquot}.

\begin{thm}\label{thm:relquot-to-weightedstablemaps}
    We have an equality
    \[ [\Qbar_{g, n}^{\ve, nrt, k}(\Gr(r, N), d)] =  \sum_{{\substack{{k_1 + \ldots + k_r = k}\\{k_i \geq 0 \,\forall \,i}}} } \bbL^{\sum_{i=1}^r(i-1)k_i}\left[\frac{\Mbar_{g, n|\ve^k}^{nrt}(\Gr(r, N), d-k)}{\bbS_{k_1} \times \bbS_{k_2} \times \cdots \times \bbS_{k_r}} \right] \]
    in $K_0(\cat{Var})$. The same equality holds for $\bbS_n$-equivariant Serre characteristics.
\end{thm}
The remainder of this section will be devoted to the proof of Theorem \ref{thm:relquot-to-weightedstablemaps}. Parallel to Lemma \ref{lem:smooth_qstrata}, the first step is to express the strata $\Qbar_{g, n}^{\ve, nrt, k}(\Gr(r, N), d)$ in terms of relative Quot schemes via the saturation map $\Qbar_{g, n}^{\ve, nrt, k}(\Gr(r, N), d)\to \Mbar_{g,n}^{nrt}(\Gr(r, N), d-k)$. Unlike the case of smooth curves, the universal curves on the domain and the target are different in general: the map contracts bivalent, unmarked rational components on which the quotient sheaf from the quasimap has degree zero saturation. Therefore, we are required to consider stacks of prestable maps. We set up the following notation.

\begin{defn}
    Let $\Mps^{nrt}_{g, n}(\Gr(r, N), d)$ denote the Artin stack of degree-$d$ maps from $n$-pointed prestable curves without rational tails to $\Gr(r, N),$ with universal curve denoted as \[\Cps^{nrt}_{g, n}(\Gr(r, N), d) \to \Mps_{g, n}
 ^{nrt}(\Gr(r, N), d).\] Again we will consider the open subset
 \(\Cps^{nrt, \circ}_{g, n}(\Gr(r, N), d) \subset \Cps^{nrt}_{g, n}(\Gr(r, N), d) \) as the complement of the markings and the nodes.

 Emulating the construction of the previous section, we let $ev^*\calS$ denote the rank-$r$ vector bundle on $\Cps^{nrt, \circ}_{g, n}(\Gr(r, N), d)$ obtained by pulling back the tautological bundle on $\Gr(r, N)$.
 
 Let \(\Quot^{k, \ve}_{\Cps^{\circ}/\Mps^{nrt}_{g, n}(\Gr(r,N),d)}(ev^*\calS)\) denote the relative Quot scheme of zero-dimensional degree-$k$ quotients on the family of curves \(\Cps^{nrt, \circ}_{g, n}(\Gr(r, N), d)\to \Mps^{nrt}_{g, n}(\Gr(r, N), d)\) which satisfies\footnote{See Definition \ref{defn:quotve} for the definition of $\ell(p).$} $\ell(p) \leq 1/\ve$ for all points $p.$ It is representable over $\Mps^{nrt}_{g, n}(\Gr(r, N), d)$ and is an Artin stack.
\end{defn}

\begin{defn}\label{defn:quotfrak}
    Define
 \[\Quot^{k, \ve, stab}_{\Cps^{\circ}_{g, n}/\Mps^{nrt}_{g, n}(\Gr(r,N),d)}(ev^*\calS) \subset \Quot^{k, \ve}_{\Cps^{\circ}/\Mps^{nrt}_{g, n}(\Gr(r, N), d)}(ev^*\calS) \]
 be the closed substack parametrizing tuples \[\left([f: C \to \Gr(r, N)] \in \Mps^{nrt}_{g, n}(\Gr(r, N), d), \mathcal{O}^{N}_C\to \mathcal{Q}\to 0\right),\] which satisfies the following stability condition: if $f$ contracts a rational bridge $T \subset C$, then $\calQ$ has positive degree along $T$. 
\end{defn}

\begin{notn}In the rest of this section, we fix the degree $d$ throughout. For a rank-$r$ vector bundle $\calE$ on $\Cps^{\circ}$, we set
\[ \Quot^{k, \ve}_{\Cps^\circ/\Mps}(\calE) := \Quot^{k, \ve}_{\Cps^{\circ}/\Mps^{nrt}_{g, n}(\Gr(r, N), d-k)}(\calE)\] for ease of notation. Similar for $\Quot^{k, \ve, stab}_{\Cps^\circ/\Mps}(\calE).$\end{notn}
   
 \begin{lem}\label{lem:quot_strata_stacky}
     For any $k$ with $0\leq k \leq d$, there is an isomorphism
     \[ \Qbar_{g, n}^{\ve, nrt, k}(\Gr(r, N), d) \cong \Quot^{k,\ve, stab}_{\Cps^{\circ}/\Mps}(ev^*\calS). \]
 \end{lem}
 \begin{proof}
     We observe that the $\ve$-quasimap stability matches with the definition of $\Quot^{k,\ve}$ together with the stability condition in Definition \ref{defn:quotfrak}. The rest of the proof is identical to that of Lemma \ref{lem:smooth_qstrata}.
 \end{proof}
 To study the motive of the relative Quot stack above, we will compare it with the trivial bundle, as in Lemma \ref{lem:indep-Quot}.
\begin{lem}\label{lem:stacky_vb_independence}
    For any $k$ with $0\leq k \leq d$, there is an equality
     \[ [\Quot^{k, \ve, stab}_{\Cps^{\circ}/\Mps}(ev^*\calS)] = [\Quot^{k, \ve, stab}_{\Cps^{\circ}/\Mps}(\calO^{\oplus r})] \]
     in $K_0(\cat{Var})$. The same equality holds for $\bbS_n$-equivariant Serre characteristics.
\end{lem}
\begin{proof}
    We observe that the statement of Lemma \ref{lem:indep-Quot} still holds for a family of smooth curves over an Artin stack as a piecewise isomorphism relative to the same base. More precisely, let $C\to S$ a flat family of curves over an Artin stack $S,$ and let $\calE$ be a vector bundle on $C.$ By passing to smooth charts of $C$ and $S$ if necessary, there still exist finite locally closed stratifications \[S = \bigsqcup_{i\in I} S_i, \pi^{-1}(S_i) = \bigsqcup_{j\in \calA_i} W_j\] satisfying the same requirements as in the proof of Lemma \ref{lem:indep-Quot}. Then the fibers of $\Quot_{C/S}^{k,\ve}(\calE)$ and $\Quot_{C/S}^{k,\ve}(\calO^{r})$ over each stratum in \begin{equation}\label{eqn:SymW}
        \bigsqcup_{(k_i\geq 0): \sum_{ i = 1}^\ell k_i = k} \ \Sym_{W_1/S}^{k_1} \times_S \cdots \times_S \Sym_{W_\ell/S}^{k_\ell}\subset \Sym^k_{C/S}
    \end{equation} are isomorphic.
    There is a natural embedding
\[ \Mbar_{g, n|\ve^k}^{nrt}(\Gr(r, N), d-k) \hookrightarrow \mathrm{Conf}_{\ve}^{k} (\Cps^{nrt, \circ}_{g, n}(\Gr(r, N), d-k)/\Mps^{nrt}_{g, n}(\Gr(r, N), d-k) ).\] The stability condition in Definition \ref{defn:quotfrak} leads to the following fiber diagram \[ \begin{tikzcd}
    &\Quot^{k, \ve, stab}_{\Cps^\circ/\Mps}(\calE) \arrow[d] \arrow[r] &\Quot^{k, \ve}_{\Cps^\circ/\Mps}(\calE) \arrow[d]\\
    &\frac{\Mbar_{g, n|\ve^k}^{nrt}(\Gr(r, N), d-k)}{\bbS_k} \arrow[r] & \frac{\mathrm{Conf}_{\ve}^{k} (\Cps^{nrt}/\Mps^{nrt}_{g, n}(\Gr(r, N), d-k) )}{\bbS_k}
\end{tikzcd} \]

    The moduli space $\frac{\Mbar_{g, n|\ve^k}^{nrt}(\Gr(r, N), d-k)}{\bbS_k}$ is stratified by its intersections with the strata from (\ref{eqn:SymW}). The fiber diagram implies that the fibers of $\Quot^{k, \ve, stab}_{\Cps^\circ/\Mps}(\calE)$ and $\Quot^{k, \ve, stab}_{\Cps^\circ/\Mps}(\calO^r)$ over the strata in $\frac{\Mbar_{g, n|\ve^k}^{nrt}(\Gr(r, N), d-k)}{\bbS_k}$ are isomorphic. Summing across these strata, we get the desired equality of Grothendieck ring classes, which implies the statement about $\bbS_n$-equivariant Serre characteristics by the proof of Corollary \ref{cor:Qstrat_over_smooth}.
\end{proof}
Using the previous two lemmas, we can now prove Theorem \ref{thm:relquot-to-weightedstablemaps}.

\begin{proof}[Proof of Theorem \ref{thm:relquot-to-weightedstablemaps}]


By fixing a flag of the trivial bundle $\calO^{\oplus r}$, we can emulate the proof of Proposition \ref{prop:relquot} to obtain a stratification
\[\Quot^{k, \ve}_{\Cps^\circ/\Mps}(\calO^{\oplus r}) = \coprod_{\substack{{\ell_1 + \cdots + \ell_r = k}\\{\ell_i \geq 0}}} \Quot^{\ell_1, \ldots, \ell_r, \ve}_{\Cps^\circ /\Mps}(\calO^{\oplus r}) \]
where 
\[\Quot^{\ell_1, \ldots, \ell_r, \ve}_{\Cps^\circ /\Mps}(\calO^{\oplus r}) \to \mathrm{Conf}_{\ve}^{k}(\Cps^{nrt, \circ}_{g, n}(\Gr(r, N), d-k)/\Mps^{nrt}_{g, n}(\Gr(r, N), d-k)) \]
is an iterated vector bundle as in Proposition \ref{prop:relquot}. We intersect the stratification with the stable locus and obtain 
\[ \Quot^{k, \ve, stab}_{\Cps^\circ/ \Mps}(\calO^{\oplus r}) = \coprod_{\ell_1, \ldots, \ell_r} \Quot^{\ell_1, \ldots, \ell_r, \ve,stab}_{\Cps^\circ/ \Mps}(\calO^{\oplus r}). \]
Refining the fiber diagram in the proof of Lemma \ref{lem:stacky_vb_independence}, each stratum fits into the following fiber diagram:
\[ \begin{tikzcd}
    &\Quot^{\ell_1, \ldots, \ell_r, \ve, stab}_{\Cps^\circ/\Mps}(\calO^{\oplus r}) \arrow[d] \arrow[r] &\Quot^{\ell_1, \ldots, \ell_r, \ve}_{\Cps^\circ/\Mps}(\calO^{\oplus r}) \arrow[d]\\
    &\frac{\Mbar_{g, n|\ve^k}^{nrt}(\Gr(r, N), d-k)}{\bbS_{\ell_1} \times \cdots \times \bbS_{\ell_r}} \arrow[r] & \frac{\mathrm{Conf}_{\ve}^{k} (\Cps^{nrt, \circ}_{g, n}(\Gr(r, N), d-k)/\Mps^{nrt}_{g, n}(\Gr(r, N), d-k) )}{\bbS_{\ell_1} \times \cdots \times \bbS_{\ell_r}}
\end{tikzcd}. \]
Since the right vertical arrow is an iterated vector bundle, so is the left vertical arrow. Summing over all tuples $\ell_1, \ldots, \ell_r$ and reindexing as in Proposition \ref{prop:relquot}, we obtain
\[[\Quot^{k, stab}_{\Cps^\circ/\Mps}(\calO^{\oplus r})] = \sum_{\substack{{k_1 + \cdots +k_r = k}\\{k_i \geq 0\, \forall\,i}}} \bbL^{\sum_{i = 1}^r (i-1)k_i} \left[ \frac{\Mbar_{g, n|\ve^k}^{nrt}(\Gr(r, N), d-k)}{\bbS_{k_1} \times \cdots \times \bbS_{k_r}}  \right].  \]
Now the proof of the equality in the Grothendieck ring is complete by combining Lemma \ref{lem:quot_strata_stacky} and Lemma \ref{lem:stacky_vb_independence}. The statement about $\bbS_n$-equivariant Serre characteristics follows from the discussion in the proof of Corollary \ref{cor:Qstrat_over_smooth}.
\end{proof}

\begin{rem}
    The presence of rational bridges in quasimaps is analogous to the destabilization (also known as expanded degeneration) encountered in (relative) logarithmic Quot schemes of curves \cite{logquot}. One expects the same shape of the formula above expresses the motive of logarithmic Quot schemes in terms of their logarithmic Hilbert schemes and powers of $\bbL.$
\end{rem}

\subsection{Wall-crossing for the no-rational-tails loci} \label{subsec:wcnrt}Specializing to $\bbS_n$-equivariant Serre characteristics, we now develop the generating function form of Theorem \ref{thm:relquot-to-weightedstablemaps}.

\begin{notn}
    Parallel to the notation $\overline{\cat{M}}_{g,r,N}$ from Definition \ref{defn:catMbar}, set \[\overline{\cat{Q}}_{g,r, N}^{nrt, \ve} = \sum_{n, d} \cat{e}^{\bbS_n}(\Qbar_{g, n}^{\ve,nrt}(\Gr(r, N), d))\quad\mbox{and}\quad \overline{\cat{Q}}_{g,r, N}^{nrt, 0} = \sum_{n, d} \cat{e}^{\bbS_n}(\Qbar_{g, n}^{nrt}(\Gr(r, N), d)).\]
\end{notn}

\begin{cor}\label{cor:complicated_wall_crossing}
    For any genus $g \geq 0$ and $\ve > 0$, we have
    \[ \overline{\cat{Q}}_{g,r, N}^{nrt, \ve} = \Delta^{r-1}_2 (\Delta \overline{\cat{M}}_{g,r, N}^{nrt} \circ_2 \Exp_{\leq M_{\ve}}(p_1))|_{\substack{p_j^{(1)}\mapsto p_j \,\forall\,j\\p_j^{(k)} \mapsto (\bbL^{k-2}q)^j\, \forall j>0,k>1}} \]
    and
    \[ \overline{\cat{Q}}_{g,r, N}^{nrt, 0} = \Delta^{r-1}_2 (\Delta \overline{\cat{M}}_{g,r, N}^{nrt} \circ_2 \Exp(p_1))|_{\substack{p_j^{(1)}\mapsto p_j \,\forall\,j\\p_j^{(k)} \mapsto (\bbL^{k-2}q)^j\, \forall j>0,k>1}} \]
\end{cor}
\begin{proof}
First, by (\ref{eqn:delta_2}) and Corollary \ref{cor:serre_char_collisions}, we have that
\[\Delta^{r-1}_2 (\Delta \overline{\cat{M}}_{g,r, N}^{nrt} \circ_2 \Exp_{\leq M_{\ve}}(p_1)) = \sum_{n,k,d \geq 0}\sum_{\substack{k_1+ \cdots + k_{r} = k\\k_i \geq 0\,\forall\,i}} \cat{e}^{\bbS_n \times \bbS_{k_1} \times \cdots \times \bbS_{k_{r}}}(\Mbar_{g, n|\ve^{k}}^{nrt}(\Gr(r, N), d))q^d. \]
Now we have
\begin{align*}\cat{e}^{\bbS_n \times \bbS_{k_1} \times \cdots \times \bbS_{k_{r}}}(\Mbar_{g, n|\ve^{k}}^{nrt}(\Gr(r, N), d))&|_{\substack{p_j^{(1)} \mapsto p_j \forall j \geq 1\\p_j^{(k)} \mapsto (\bbL^{k-2}q)^j\,\forall j\geq1, k > 1}} \\&= \cat{e}^{\bbS_n}\left(\frac{\Mbar_{g, n|\ve^{k}}^{nrt}(\Gr(r, N), d)}{\bbS_{k_1} \times \cdots \times\bbS_{k_r}}\right) q^{k} \cdot \prod_{i= 1}^{r}\bbL^{(i - 1) k_i}.
\end{align*}
by (\ref{eqn:multisymmetric_invariants}). The proof is complete by Theorem \ref{thm:relquot-to-weightedstablemaps}.
\end{proof}

The composition of the coproduct and substitutions above can be succintly expressed as the following operators from the introduction.

\begin{defn}
    Let \(\boldsymbol{B}_{r, \ve} : K_0(\cat{MHS})\otimes \Lambda \to K_0(\cat{MHS})\otimes \Lambda \) be the $\QQ$-algebra automorphism determined by \[\boldsymbol{B}_{r, \ve}: p_j \mapsto p_j + \sum_{k = 1}^{M_{\ve}} \qbinomopt[\bbL^j]{r + k - 1}{k} \ q^{jk}.\]  As its $\ve \to 0^+$ limit, we define
\[ \boldsymbol{B}_{r, 0}: p_j \mapsto p_j + \prod_{i = 0}^{r-1} \frac{1}{1 - \left(\bbL^{i}q\right)^j} - 1. \]
\end{defn}

\begin{cor}\label{cor:nrtqvsm}
For any genus $g \geq 0$ and fixed $\ve > 0$, we have
    \[\boldsymbol{B}_{r, \ve}\left( \overline{\cat{M}}_{g,r, N}^{nrt} \right) =  \overline{\cat{Q}}_{g,r, N}^{\ve,nrt}. \]
    In the $\ve \to 0^+$ limit, we have
    \[\boldsymbol{B}_{r, 0}\left( \overline{\cat{M}}_{g,r, N}^{\ve,nrt} \right) = \overline{\cat{Q}}_{g,r, N}^{0,nrt} \]
\end{cor}

\begin{proof}
For any integer $n > 0$, we have
    \begin{align*}
        \Delta^{r-1}_2(\Delta p_n\circ_2 \Exp_{\leq M_{\ve}}(p_1)) &=  \Delta_2^{r-1}(p_n^{(1)} + \left(\Exp_{\leq M_{\ve}}(p_n )\right)^{(2)}) \\&= \Delta^{r-1}_2\left(p_n^{(1)} + \left( \sum_{k = 1}^{M_{\ve}}  h_{k} \circ p_n \right)^{(2)}\right) \\&= p_n^{(1)} +  \sum_{k = 1}^{M_{\ve}}  (h_{k} \circ p_n)|_{p_j \mapsto p_j^{(2)} + \cdots + p_j^{(r+1)}\,\forall \,j}. 
    \end{align*}  
    Therefore, writing $\lambda_i(\sigma)$ for the number of $i$-cycles in a permutation $\sigma$, we find that
    \begin{align*}
        (\Delta^{r-1}_2(\Delta p_n\circ_2 \Exp_{\leq M_{\ve}}(p_1)))&|_{p_{j}^{(k)} \mapsto (\bbL^{k-2}q)^j\,\forall \,k>1,\, p_j^{(1)} \mapsto p_j} \\&= p_n +  \sum_{k = 1}^{M_{\ve}}  (h_{k} \circ p_n)|_{p_j \mapsto q^j(1 + \bbL^{j} + \cdots + \bbL^{j(r-1)})\,\forall \,j}  \\&= p_n + \sum_{k = 1}^{M_{\ve}} \frac{1}{k!} \sum_{\sigma \in \bbS_k} \prod_{i > 0} (q^{ni}(1 + \bbL^{ni} + \cdots +\bbL^{ni(r - 1)}))^{\lambda_i(\sigma)} \\&= p_n + \sum_{k = 1}^{M_{\ve}} q^{nk}\psi_n\left(h_k \circ(1 + \bbL + \cdots + \bbL^{r-1}\right))
        \\&= p_n + \sum_{k = 1}^{M_{\ve}} q^{nk}\psi_n\left(\cat{e}(\Sym^{k}(\P^{r - 1}))\right)
        \\&= p_n + \sum_{k = 1}^{M_{\ve}} q^{nk} \qbinomopt[\bbL^n]{r + k - 1}{k}
    \end{align*}  
    which, together with Corollary \ref{cor:complicated_wall_crossing}, proves the first part of the statement. Now consider the $\ve \to 0^+$ limit, which corresponds to replacing $\Exp_{\leq M_\ve}(p_1)$ with $\Exp(p_1)$ and hence amounts to taking $M_{\ve} \to \infty$. In the limit we obtain the substitution
    \[ p_n \mapsto p_n + \psi_n\left(\sum_{k = 1}^{\infty} \cat{e}(\Sym^{k}(\P^{r-1})) q^{k} \right), \]
    and 
    \[\psi_n\left(\sum_{k = 1}^{\infty} \cat{e}(\Sym^{k}(\P^{r-1})) q^{k} \right) = \prod_{i = 0}^{r - 1}\frac{1}{1 - (\bbL^i q)^n} - 1, \]
    by Lemma \ref{lem:sym_of_Pr}. Combining this observation with Corollary \ref{cor:complicated_wall_crossing} finishes the proof.
\end{proof}

\begin{rem}\label{rem:punctualquot}
    The motive of $\Sym^{k}(\P^{r-1})$ appearing in the above formulas agrees with that of the length-$k$ punctual Quot scheme of a rank-$r$ locally free sheaf on a smooth curve by work of Ricolfi \cite[Remark 3.4]{Ric20}. Therefore, the operators $\boldsymbol{B}_{r,\ve}$ and $\boldsymbol{B}_{r,0}$ can be interpreted as substituting markings (with their permutation action) with punctual Quot schemes of $ev^*\calS$ in a precise way. It would be interesting to see whether the formulas in this work interact with the plethystic exponential formulas \cite[Theorem A]{Ric20} of his work.
\end{rem}

\section{Genus-zero maps and quasimaps}\label{sec:g0maps}
In this section, we deduce Theorem \ref{thm:main} from Corollary \ref{cor:nrtqvsm} using plethystic formulas relating the equivariant Serre characteristics of the graded $\bbS$-spaces $\Qbar^{\ve}_{g,r,N}$ and $\Qbar^{\ve, nrt}_{g,r,N},$ as well as the graded $\bbS$-spaces $\Mbar_{g,r,N}$ and $\Mbar_{g,r,N}.$ The formalism of plethysms of graded $\bbS$-modules has been defined in §\ref{subsec:plethysmgraded}, in particular Definition \ref{defn:plethysmgraded}.

The plethystic formula comes from the combinatorial operation of attaching rooted rational tails to curves with no rational tails to recover all points in $\Mbar_{g,n}(\Gr(r,N),d)$ and $\Qbar^{\ve}_{g,n}(\Gr(r,N),d)$).

\subsection{Moduli spaces of pointed maps}

We begin by introducing the moduli space corresponding to the rooted rational tails. Recall that due to the transitive $\mathrm{GL}_N$-action on $\Gr(r,N),$ the evaluation maps $ev_{n+1}: \Mbar_{g,n+1}(\Gr(r,N),d)\to \Gr(r,N)$ and $ev_{n+1}: \Qbar^{\ve}_{g,n+1}(\Gr(r,N),d)\to \Gr(r,N),$ which are $\mathrm{GL}_r$-equivariant, are Zariski locally trivial fibrations. The same holds for the evaluation maps $ev_{n+1}: \M_{g,n+1}(\Gr(r,N),d)\to \Gr(r,N)$ and $ev_{n+1}: \Q^{\ve}_{g,n+1}(\Gr(r,N),d)\to \Gr(r,N).$

\begin{defn}
    Let $\Box$ denote any one of $\Mbar, \Qbar^{\ve}, \M,$ or $\Q^{\ve}.$ We define \[\Box^\star_{g,n}(\Gr(r,N),d)\subset \Box_{g,n+1}(\Gr(r,N),d)\] as a fiber of $ev_{n+1}: \Box_{g,n+1}(\Gr(r,N),d)\to \Gr(r,N).$

    Let $\Box_{g,r,N}^\star$ be the graded $\bbS$-space defined by $\Box_{g,r,N}^\star(n,d) :=\Box_{g,n}^\star(\Gr(r,N),d).$ Recall from the introduction that their equivariant Serre characteristics are denoted as $\overline{\cat{M}}^\star_{g,r,N},$ $\overline{\cat{Q}}^{\ve,\star}_{g,r,N},$ $\cat{M}^\star_{g,r,N},$ and $\cat{Q}^{\ve,\star}_{g,r,N}$ respectively.
\end{defn}

The Serre characteristics of the pointed moduli spaces are related to those of $\Box_{g,r,N}$ by the following formula.

\begin{lem}\label{lem:MQstar}
    We have \[\mathsf{e}^{\bbS}(\Box_{g,r,N}^{\star}) = \frac{1}{\mathsf{e}(\Gr(r,N))}\cdot \frac{\partial}{\partial p_1}\mathsf{e}^{\bbS}(\Box_{g,r,N}).\]
\end{lem}
\begin{proof}
    This follows from an $\bbS_n$-equivariant Zariski locally trivial fibration \[\mathrm{Res}^{\bbS_{n+1}}_{\bbS_{n}}\left(\Box_{g,n+1}(\Gr(r,N),d)\right)\to \Gr(r,N),\] where $\Gr(r,N)$ carries trivial $\bbS_n$-action, and the fiber is $\Box_{g,n}^\star(\Gr(r,N),d).$
\end{proof}

\subsection{Plethysm and rational tails}

The graded $\bbS$-spaces $\Mbar_{0,r,N}^{\star}$ resp. $\Qbar_{0,r,N}^{\ve, \star}$ account for the rational tail contribution to $\Mbar_{g,r,N}$ resp. $\Qbar_{g,r,N}^{\ve}.$ Extending the discussion on $ev_{n+1},$ the forgetful map that restricts a stable map to the complement of the rational tails is, up to a finite group quotient, a Zariski locally trivial fibration with fiber as strata in $\Mbar_{0,r,N}^\star$ resp. $\Qbar_{0,r,N}^\star.$ This leads to the following formulas, the proof of which is entirely parallel to §3.4 of \cite{ks-genus1}.

\begin{lem}\label{lem:nrtplethysm}
   Let $\Box$ denote either $\Mbar$ or $\Qbar^{\ve}.$ For any genus $g \geq 1$ and $\ve \geq 0$, we have \begin{align*}
        \mathsf{e}^{\bbS}(\Box_{g,r,N}) & = \mathsf{e}^{\bbS}(\Box^{nrt}_{g,r,N})\circ(p_1 + \mathsf{e}^{\bbS}(\Box_{0,r,N}^\star))\\ & = \mathsf{e}^{\bbS}(\Box^{nrt}_{g,r,N})\circ\left(p_1 + \frac{1}{\mathsf{e}(\Gr(r,N))}\frac{\partial}{\partial p_1}\mathsf{e}^{\bbS}(\Box_{0,r,N})\right).
    \end{align*}
\end{lem}

Note that the condition $g \geq 1$ is crucial to Lemma \ref{lem:nrtplethysm}, because it relies on the dual graph of every source curve to have a well-defined \textit{core}, which is the maximal subgraph of genus $g$; see the proof of \cite[Proposition 3.9]{ks-genus1}. The genus-zero case is discussed in more detail below in Remark \ref{rem:genus_zero}. To translate the comparison formula between the no-rational-tail moduli spaces (Theorem \ref{thm:relquot-to-weightedstablemaps}) to one between $\Qbar^{\ve}_{g,r,N}$ and $\Mbar_{g,r,N},$ we use the following plethystic inversion formula. It is a straightforward consequence of the previous work of Getzler--Pandharipande on $\Mbar_{0,n}(\P^r,d)$ \cite{GetzlerPandharipande}.

\begin{lem} \label{lem:g0plethysm}We have
    \[(p_1+\overline{\cat{M}}_{0,r,N}^\star)\circ (p_1 - \cat{M}_{0,r,N}^\star) = (p_1+\overline{\cat{M}}_{0,r,N}^\star)\circ (p_1 - \cat{M}_{0,r,N}^\star) = p_1\]as well as the parallel formula regarding $\overline{\cat{Q}}^{\ve, \star}_{0, r,N}$ and $\cat{Q}^{\ve,\star}_{0,r,N}.$
\end{lem}
\begin{proof}
    We observe that both $\Mbar_{0,n}(\Gr(r,N),d)$ and $\Qbar^{\ve}_{0,n}(\Gr(r,N),d)$ are recursively stratified by decorated trees in the same way as $\Mbar_{0,n}(\P^r,d),$ with $\Gr(r,N)$ standing in for $\P^r.$ Therefore, the discussion leading up to Corollary 4.6 of \cite{GetzlerPandharipande} applies verbatim to show that \[\overline{\cat{M}}_{0,r,N}^\star = \cat{M}_{0,r,N}^\star\circ(p_1+\overline{\cat{M}}_{0,r,N}^\star),\] and similar for the pair $(\Qbar^{\ve, \star}_{0, r,N},\Q^{\ve,\star}_{0,r,N}).$ From this, we compute \begin{align*}
        (p_1 - \cat{M}_{0,r,N}^\star)\circ (p_1+\overline{\cat{M}}_{0,r,N}^\star) & = \left(p_1 + \overline{\cat{M}}_{0,r,N}^\star\right) - \cat{M}_{0,r,N}^\star\circ (p_1+\overline{\cat{M}}_{0,r,N}^\star)\\ & = p_1 + \overline{\cat{M}}_{0,r,N}^\star - \overline{\cat{M}}_{0,r,N}^\star=p_1.
    \end{align*}

    This implies the other inversion formula for general reasons, which we now explain. The endomorphisms $\Phi_{\M},\Phi_{\Mbar}: K_0(\mathsf{MHS})\otimes \Lambda[\![q]\!]\to K_0(\mathsf{MHS})\otimes \Lambda[\![q]\!]$ defined by \[\Phi_{\M}(g):=(p_1 - \cat{M}_{0,r,N}^\star)\circ g\text{ and } \Phi_{\Mbar}(g):=(p_1 + \overline{\cat{M}}_{0,r,N}^\star)\circ g\] satisfy $(\Phi_{\M}\circ \Phi_{\Mbar})(g) = p_1\circ g = g,$ so $ \Phi_{\M}\circ \Phi_{\Mbar}= \mathrm{id}_{K_0(\mathsf{MHS})\otimes \Lambda[\![q]\!]}.$ From this, we know that $\Phi_{\M}$ is the two-sided inverse to $\Phi_{\Mbar},$ namely $\Phi_{\Mbar}\circ \Phi_{\M}= \mathrm{id}.$ Evaluating this equality on $p_1,$ we get $(p_1 + \overline{\cat{M}}_{0,r,N}^\star)\circ (p_1 - \cat{M}_{0,r,N}^\star) = p_1$ as desired.
\end{proof}
\begin{rem}
    So far, the formulas in this section hold in the Grothendieck ring of graded $\bbS$-varieties $K_0(\mathsf{Var}, \bbS)[\![q]\!],$ where the operator $\frac{\partial}{\partial p_1}$ is replaced by the operator $D: K_0(\mathsf{Var}, \bbS)[\![q]\!]\to K_0(\mathsf{Var}, \bbS)[\![q]\!]$ defined by $(D\mathcal{X})(n,d):=(\Res^{\bbS_{n+1}}_{\bbS_n}\mathcal{X}(n+1,d)).$
\end{rem}

We now finish the proof of the main comparison formula. See (\ref{eqn:Bre}) and (\ref{eqn:Br0}) for the definition of the algebra automorphisms $\boldsymbol{B}_{r, \ve}$ and $\boldsymbol{B}_{r, 0}.$
\begin{thm}[Theorem \ref{thm:main}]
    For any $\ve > 0$ and $\ve = 0^+$, we have
\[\boldsymbol{B}_{r, \ve}\left[ \overline{\cat{M}}_{g, r, N} \circ \left(p_1 - \frac{\partial}{\partial p_1}\frac{\cat{M}_{0, r, N}}{\cat{e}(\Gr(r, N))}\right)\right] = \overline{\cat{Q}}_{g, r, N}^{\ve} \circ \left(p_1 - \frac{\partial}{\partial p_1}\frac{\cat{Q}_{0, r, N}^{\ve}}{\cat{e}(\Gr(r, N))}\right).\]
\end{thm}
\begin{proof}
    Lemma \ref{lem:g0plethysm} inverts the formula in Lemma \ref{lem:nrtplethysm} to give \[\mathsf{M}^{nrt}_{g,r,N} = \overline{\cat{M}}_{g,r,N}\circ (p_1 - \mathsf{M}^{\star}_{0,r,N}) \stackrel{\text{(\ref{lem:MQstar})}}{=}\overline{\cat{M}}_{g,r,N}\circ \left(p_1 - \frac{\partial}{\partial p_1}\frac{\cat{M}_{0, r, N}}{\cat{e}(\Gr(r, N))}\right),\text{ and }\]
    \[\mathsf{Q}^{\ve, nrt}_{g,r,N} = \overline{\cat{Q}}^{\ve}_{g,r,N}\circ (p_1 - \mathsf{Q}^{\ve,\star}_{0,r,N}) \stackrel{\text{(\ref{lem:MQstar})}}{=}\overline{\cat{Q}}^{\ve}_{g,r,N}\circ \left(p_1 - \frac{\partial}{\partial p_1}\frac{\cat{Q}_{0, r, N}^{\star}}{\cat{e}(\Gr(r, N))}\right).\] The result follows from plugging the above two formulas into Corollary \ref{cor:nrtqvsm}.
\end{proof}
\subsection{Serre characteristics of genus-zero maps}
For completeness, we recall previous work of Bagnarol \cite{Bagnarol} on $\cat{M}_{0,r,N}$ and explain how they can be adapted to determine $\cat{Q}^{\ve}_{0,r,N}.$ Combining these with Lemma \ref{lem:MQstar} leads to formulas of $\cat{M}^\star_{0,r,N}$ and $\cat{Q}^{\ve,\star}_{0,r,N}.$

\begin{defn}
Let $\mathrm{Quot}_{(N-r),d}:=\mathrm{Quot}^{(t+1)(N-r)+d}_{\P^1}(\mathcal{O}^{N})$ be the Quot scheme of rank-$(N-r),$ degree-$d$ quotients of $\mathcal{O}_{\P^1}^{N},$ which have Hilbert polynomial $(t+1)(N-r)+d\in \mathbb{Q}[t].$
\end{defn}

The space of parametrized maps $\mathrm{Map}_d(\P^1, \Gr(r,N))$ is an open subscheme of $\mathrm{Quot}_{(N-r),d}$ where the quotient sheaf is locally free. We also have $\mathrm{Quot}_{(N-r),d}/\mathrm{PGL}_2\cong \Q_{0,0}^{0^+}(\Gr(r,N),d),$ the moduli space of stable quotients on a unmarked rational curve.

The Grothendieck ring class of $\mathrm{Quot}_{(N-r),d}$ has been determined by Strømme via its Białynicki--Birula cell decomposition coming from a torus action.

\begin{thm}\cite[Theorem 2.1]{Stromme}
Let $m_{d,i}$ be the number of triples $(\mathbf{a}, \mathbf{b}, \mathbf{c})\in \mathbb{Z}^r\times \mathbb{Z}^{r+1}\times \mathbb{Z}^r$ that satisfy \[b_0 = 0\leq a_1\leq b_1\leq a_2\leq \cdots\leq b_{s-1}\leq a_s\leq d = b_s,\] \[0\leq c_1\leq \cdots\leq c_s\leq N-r,\] \[\sum_{j=1}^r (a_j  + c_j(1+b_j-b_{j-1}))=i,\text{ then }\] \[[\mathrm{Quot}_{(N-r),d}] = \sum_{i = 0}^{Nd + r(N-r)}m_{d,i}\mathbb{L}^i\in K_0(\mathsf{Var}).\]
\end{thm}

\begin{notn}
We set the following conventions.
  \begin{enumerate}
    \item Let $\mathsf{Quot}_{N-r, \mathbb{P}^1}$ be the generating function $\sum_{d=0}^\infty [\mathrm{Quot}_{(N-r),d}]q^d.$
    \item Let $\mathsf{Quot}_{0, \mathbb{P}^1}$ be the generating function $\sum_{k = 0}^\infty [\mathrm{Quot}^k_{\P^1}(\mathcal{O}^r)]q^k.$ See Theorem \ref{thm:quotmotive} for the formula of $[\mathrm{Quot}^k_{\P^1}(\mathcal{O}^r)]\in K_0(\mathsf{Var})$.
    \item Following \cite{GetzlerPandharipande}, let $\mathsf{F}(\mathbb{P})$ denote the $\mathbb{S}$-space $\mathsf{F}(\mathbb{P})(n):=\mathrm{Conf}^n(\P^1).$ Theorem 3.2 of loc. cit. determines $\cat{e}^{\bbS}(\mathsf{F}(\mathbb{P})).$
  \end{enumerate}
\end{notn}

Stratifying the boundary of $\mathrm{Map}_d(\P^1, \Gr(r,N))\subset \Quot_{(N-r),d}$ by the torsion of the quotient sheaf, Bagnarol compares their Grothendieck ring classes.
\begin{thm}\label{thm:bagnarol} \cite[Theorem 4.14]{Bagnarol} We have the following formulas in $K_0(\mathsf{MHS})\otimes \Lambda[\![q]\!]$:
  \[\mathsf{M}_{0,r,N} = \frac{\mathsf{e}^{\bbS}(\mathsf{F}(\mathbb{P}))}{\mathsf{e}(\mathrm{PGL}_2)} \mathsf{Quot}_{N-r, \mathbb{P}^1} \cdot (\mathsf{Quot}_{0, \mathbb{P}^1})^{-1}.\]
\end{thm}

Applying Corollary \ref{cor:Qstrat_over_smooth} in $g = 0$ and taking equivariant Serre characteristics as in §\ref{subsec:wcnrt}, we have

\begin{lem}\label{lem:genuszero-wall-crossing}
    \[\mathsf{Q}^{\varepsilon}_{0,r,N} = \boldsymbol{B}_{r, \ve}(\cat{M}_{0, r, N}).\]
\end{lem}
Together, Theorem \ref{thm:bagnarol} and Lemma \ref{lem:genuszero-wall-crossing} determine $\mathsf{Q}^{\varepsilon}_{0,r,N}$ for all $\ve$. Combining with Lemma \ref{lem:g0plethysm}, this implies that Theorem \ref{thm:main} is both an invertible and computationally effective transformation between $\cat{\overline{M}}_{g, r, N}$ and $\cat{\overline{Q}}_{g, r, N}^{\ve}$ for any $\ve$ and $g \geq 1$.

\begin{rem}\label{rem:genus_zero}
    Note that Theorem \ref{thm:bagnarol} and Lemma \ref{lem:genuszero-wall-crossing} together determine the generating functions $\overline{\cat{M}}_{0, r, N}^\star$ and $\overline{\cat{Q}}_{0, r, N}^{\ve,\star}$, which determines the $\bbS_{n -1}$-equivariant Serre characteristic
    \[ \cat{e}^{\bbS_{n - 1}}(\Qbar_{0, n}^{\ve}(\Gr(r, N), d)). \]
    The full $\bbS_n$-equivariant Serre characteristic can be computed following the proof of \cite[Theorem 4.5]{GetzlerPandharipande} or \cite[Theorem 3.4]{Bagnarol}; these theorems and their proofs go through unchanged for the quasimap spaces.
\end{rem}

\subsection{Contracting genus-zero maps to $\P^{N-1}$}\label{subsec:alternative_calculation}

We outline an alternative wall-crossing formula in the case of maps to $\Gr(1,N) = \P^{N-1}.$ Marian--Oprea--Pandharipande \cite[§5.2]{MOP} constructs a contraction morphism $c: \Mbar_{g,n}(\P^{N-1},d)\to \Qbar_{g,n}(\P^{N-1},d)$ that contracts unmarked rational tails to basepoints. The morphism $c$ factors through $$\Mbar_{g,n}(\P^r,d)\to \cdots\to \Qbar^{1/\delta}_{g,n}(\P^r,d)\to\Qbar^{1/(\delta+1)}_{g,n}(\P^r,d)\to \cdots\to \Qbar_{g,n}(\P^r,d),$$ in which each $\Qbar^{1/\delta}_{g,n}(\P^r,d)\to\Qbar^{1/(\delta+1)}_{g,n}(\P^r,d)$ contracts maximal rational tails of degree $\leq (\delta+1)$ as basepoints. We stratify both moduli spaces based on this contraction morphism.

\begin{defn}
    Let $\Mbar^{n_0rt}_{g, 1, N}$ be the graded $\bbS$-space as the subspace of $\Mbar_{g,1,N}$ that has no unmarked rational tail. For $\varepsilon>0,$ let $\Mbar^{n^{\varepsilon}_0rt}_{g, 1,N}$ be the subspace of $\Mbar_{g, 1,N}$ that has no unmarked rational tail of (total) degree less than or equal to $M_{\varepsilon}.$ As before, denote their equivariant Serre characteristics as $\overline{\mathsf{M}}^{n_0rt}_{g,1,N}$ and $\overline{\mathsf{M}}^{n_0^{\varepsilon}rt}_{g,1,N}$ respectively. To record contributions from unmarked rational tails, let \[\overline{\cat{M}}_{0,1,N}^\star(0):=\sum_{\delta\geq 1}\cat{e}(\Mbar_{0,0}^{\star}(\P^{N-1},\delta))q^\delta \text{ and } \overline{\cat{M}}_{0,1,N}^{\star}(0)_{\leq M_\varepsilon}:=\sum_{\delta\leq M_\varepsilon}\cat{e}(\Mbar_{0,0}^{\star}(\P^{N-1},\delta))q^\delta.\]
    
\end{defn}

Stratifying moduli spaces of stable maps and $\ve$-stable quasimaps by the degree of rational tails and lengths of basepoints respectively, we get the following plethystic formulas that are analogous to Lemma \ref{lem:nrtplethysm}.
\begin{lem}\label{lem:n0rtplethysm} We have the following formulas of equivariant Serre characteristics for $\ve>0$:
  \begin{enumerate}
    \item $\overline{\cat{M}}_{g,1,N} = \overline{\cat{M}}_{g,1,N}^{n_0^{\ve}rt}\circ (p_1 + \overline{\cat{M}}_{0,1,N}^{\star}(0)_{\leq M_\varepsilon}),$
    \item $\overline{\cat{Q}}^{\varepsilon}_{g,1,N} = \overline{\cat{M}}_{g,1,N}^{n_0^{\ve}rt}\circ \left(p_1 + \frac{q(1 - q^{M_\ve})}{1- q} \right).$
  \end{enumerate}

  The following can be seen as the $\ve\to 0^+$ limit:
  \begin{enumerate}
    \item $\overline{\cat{M}}_{g,1,N} = \overline{\cat{M}}_{g,1,N}^{n_0rt}\circ (p_1 + \overline{\cat{M}}_{0,1,N}^{\star}(0)),$
    \item $\overline{\cat{Q}}^{0}_{g,1,N} = \overline{\cat{M}}_{g,1,N}^{n_0rt}\circ \left(p_1 + \frac{q}{1-q}\right).$
  \end{enumerate}
\end{lem}

\begin{rem}
    We remark that item (2) in Lemma \ref{lem:n0rtplethysm} is a special case of Corollary \ref{cor:nrtqvsm}: it is straightforward to check that for any $f$,
    \[ f\circ \left(p_1 + \frac{q(1-q^{M_{\ve}})}{1 - q}\right) = \boldsymbol{B}_{1, \ve}(f) \text{ and }f\circ \left(p_1 + \frac{q}{1 - q}\right) = \boldsymbol{B}_{1, 0}(f).\]
\end{rem}

\begin{cor}\label{cor:alternative_wall_crossing} The equivariant Serre characteristics of $\ve$-stable quasimaps and stable maps to $\P^{N-1}$ are related by
  $$\overline{\cat{Q}}^{\ve}_{g,1,N} = \overline{\cat{M}}_{g,1,N}\circ \left(p_1 + \frac{q(1 - q^{M_\ve})}{1- q} - \overline{\cat{M}}_{0,1, N}^{\star}(0)_{\leq M_{\ve}}\right),$$
  $$\overline{\cat{Q}}^{0}_{g,1,N} = \overline{\cat{M}}_{g,1,N}\circ \left(p_1 + \frac{q}{1- q} - \overline{\cat{M}}_{0,1, N}^{\star}(0)\right).$$
\end{cor}
\begin{proof}
  For any degree $(0, d)$ element $G \in K_0(\cat{MHS}) \otimes \Lambda [\![q]\!]$ and any $F \in K_0(\cat{MHS}) \otimes \Lambda [\![q]\!]$ we have $G \circ F = G$. Therefore $(p_1 + \overline{\cat{M}}_{0,1,N}^\star(0))\circ (p_1 - \overline{\cat{M}}_{0,1,N}^\star(0))=p_1,$ and the same holds for $\overline{\cat{M}}_{0,1,N}^\star(0)_{\leq M_{\ve}}.$

  Therefore, composing the two formulas from Lemma \ref{lem:n0rtplethysm} gives $$\overline{\cat{Q}}^{0}_{g,1, N} = \overline{\cat{M}}_{g,1, N}\circ (p_1 - \overline{\cat{M}}_{0,1,N}^\star(0))\circ \left(p_1 + \frac{q}{1- q}\right),$$ $$\overline{\cat{Q}}^{\ve}_{g,1, N} = \overline{\cat{M}}_{g,1, N}\circ (p_1 - \overline{\cat{M}}_{0,1,N}^\star(0)_{\leq M_{\ve}})\circ \left(p_1 + \frac{q(1 - q^{M_\ve})}{1- q}\right).$$ Evaluating the plethysm of the last two terms gives the desired formula.
\end{proof}

\begin{rem}
    In \cite[§5.3]{MOP}, the authors comment that there does not exist a contraction map $\Mbar_{g, n}(\Gr(r, N), d)$ to $\Qbar_{g, n}(\Gr(r, N), d)$ when $r\geq 2.$ From our perspective, one obstruction for such contraction morphisms to exist is the $\bbL$-factors appearing in the operators $\boldsymbol{B}_{r, \ve}$ and $\boldsymbol{B}_{r, 0}$ which contribute to the motives of the quasimap spaces.
\end{rem}

\section{The topological Euler characteristic in genus one}\label{sec:genus_one}
 We now explain how to combine this work with graph enumeration techniques from our previous work \cite{ks-genus1} to calculate the $\bbS_n$-equivariant \textit{topological} Euler characteristic
 \[ \chi^{\bbS_n}(\Qbar_{1, n}^{\ve}(\Gr(r, N), d)) = \sum_{i}(-1)^i \ch_nH^i_c(\Qbar_{1, n}^{\ve}(\Gr(r, N), d)) \in \Lambda. \]
 
 We fix a generic $\C^\star$-action on $\Gr(r, N)$ and calculate $\bbS_n$-equivariant Serre characteristic of the fixed point locus $\Mbar_{1,n}^{nrt}(\Gr(r,N), d)^{\C^\star}.$ This determines the $\bbS_n$-equivariant topological Euler characteristic $\chi^{\bbS_n}(\Mbar_{1,n}^{nrt}(\Gr(r,N), d))$, and then the Euler characteristic $\chi^{\bbS_n}(\Mbar_{1, n}(\Gr(r, N), d))$ is determined by plethysm with the genus-zero contribution determined by Bagnarol \cite{Bagnarol}: the equation
\begin{equation}\label{eqn:genus_one_tails}
 \overline{\cat{M}}_{1, r, N} = \overline{\cat{M}}_{1, r, N}^{nrt} \circ (p_1 + \overline{\cat{M}}_{0, r, N}^\star),
 \end{equation}
 holds upon specializing to topological Euler characteristics, and Bagnarol has determined $\overline{\cat{M}}_{0, r, N}^\star$. 
 Set
 \[ \overline{\cat{M}}_{1, r, N}^{nrt, \C^\star}:= \sum_{n,d} \cat{e}^{\bbS_n}(\Mbar_{1,n}^{nrt}(\Gr(r,N), d)^{\C^\star}). \]
Once we have determined $\overline{\cat{M}}_{1, r, N}^{nrt, \C^\star}$, we can also calculate $\chi^{\bbS_n}(\Qbar_{1, n}^{\ve}(\Gr(r, N), d))$, as we make precise in Corollary \ref{cor:top_chi_genus_one} below. Sample calculations of $\chi^{\bbS_0}(\Qbar_{1,0}(\Gr(r,N),d))$ are presented in Table \ref{table:chiQ10bar} from the introduction.

We remark that the interiors $\M_{1,n}(\Gr(r,N),d)$ and $\Q^{\ve}_{1,n}(\Gr(r,N),d)$ are likely to be reducible in general \cite{Bruguieres1987}, which makes torus localization a valuable tool to study the topology of the compactified mapping spaces.

\subsection{Torus action}
A tuple of integer weights $(\lambda_1, \ldots, \lambda_N) \in \Z^{N}$ determines an action of $\C^\star$ on $\C^{N}$, which induces an action of $\C^\star$ on $\Gr(r, N)$, and on the moduli space $\Mbar_{1, n}(\Gr(r, N), d)$. For a generic choice of weights, the $\C^\star$-fixed points in $\Gr(r, N)$ are exactly the $r$-dimensional coordinate subspaces, which we identify with the set $\binom{[N]}{r}$ of size-$r$ subsets of $[N] = \{1, \ldots, N\}$. There exists a (necessarily unique) torus-invariant rational curve connecting two invariant points $S_1, S_2, \in \binom{[N]}{r}$ if and only if $|S_1 \cap S_2| = r-1$. In this way we obtain a graph $J(r, N)$ with vertex set $\binom{[N]}{r}$, which is the $1$-skeleton of the hypersimplex $\Delta(r, N)$ and is sometimes called the Johnson graph.

The fixed locus $\Mbar_{1,n}^{nrt}(\Gr(r,N), d)^{\C^\star}$ is stratified by the following graphs, see e.g. \cite{kspp}:
\begin{enumerate}
  \item the graph with a single genus one vertex, colored by a point in $\Gr(r,N)^{\C^\star},$
  \item a cycle of at least two genus-zero vertices, where each vertex is decorated by a point in $\Gr(r,N)^{\C^\star}$, and each edge is labelled by a positive integer. Two vertices may be connected by an edge only if there is an edge between the corresponding points of $\Gr(r, N)^{\C^\star}$ in $J(r, N)$. The edge weights must sum to $d$.
\end{enumerate}

It is well-known how to express the stratum corresponding to each graph as a product of moduli spaces of curves; in the no-rational-tails case, this can be adapted from \cite[\S 4.1]{ks-genus1}. 

\subsection{Walks in the Johnson graph} We now explain how to adapt \cite[\S 5]{ks-genus1} to calculate $\overline{\cat{M}}_{1, r, N}^{nrt, \C^\star}$, by adding up the contribution from each graph. The graph automorphisms and the contribution from each vertex of the graph are exactly the same as in \cite{ks-genus1}. The only difference is the structure of $\C^\star$-fixed points and $\C^\star$-invariant $\P^1$'s in $\Gr(r,N)$ as encoded in the Johnson graph $J(r, N)$, which affects the set of graphs in item (2). The set of decorated graphs in item (2) is then the set of isomorphism classes of graph homomorphisms from edge-decorated cycle graphs to $J(r, N)$. Techniques in \cite{ks-genus1} reduce this to counting walks in the Johnson graph. Let $C_k$ be a labelled $k$-cycle and let $P_k$ be a labelled path with $k$ edges and $k + 1$ vertices.

\begin{lem}
  The number of graph homomorphisms from $C_k$ to the Johnson graph $J(r, N)$ is \[\varpi_{k, (r, N)}= \sum_{j = 0}^r \left(\binom{N}{j}-\binom{N}{j-1}\right)\left((r-j)(N-r-j)-j\right)^{k}.\]

  The number of graph homomorphisms from $P_k$ to the Johnson graph is \[\omega_{k, (r, N)} = \binom{N}{r}\cdot (r(N-r))^{k-1}.\]
\end{lem}

\begin{proof}
  Let $A_{r, N}$ be the adjacency matrix of $J(r, N)$. The $ij$-th entry of $A_{r,N}^k$ is the number of paths of length $k$ between the $i$-th and $j$-th vertices. Thus, the number of homomorphisms from $C_k$ is given by $\mathrm{tr}(A_{r,N}^k).$ The eigenvalues of $A_{r,N}$ are well-known: they are $(r-j)(N-r-j)-j$ for $j = 0,\dots, r$ with multiplicity $\binom{N}{j}-\binom{N}{j-1}$ \cite{Delsarte1973}. Combining the two facts leads to the first count.

  For maps from a labeled path, the count simply comes from picking the starting vertex and then neighbors of the previous vertices.
\end{proof}

Now one substitutes $\varpi_{k, (r, N)}$ (respectively $\omega_{k, (r, N)}$) for $\varpi_k(r+1)$ (respectively $\omega_k(r+1)$) in the derivation of \cite[\S 5.5]{ks-genus1} to derive a formula for $\overline{\mathsf{M}}_{1,r,N}^{nrt, \C^\star}.$

\begin{prop}\label{prop:chi_genus_one}
Let \(\mathscr{A}_g := \sum_{n > 2 - 2g} \mathsf{e}^{S_n}(\M_{g, n}),\) and define \[\eta_{k, d}(\Gr(r,N))=\frac{1}{4}\sum_{j \mid k}  \sum_{i \mid j}  \sum_{\substack{{\ell}\\{2\ell \mid i \text{ and } k \mid d\ell}}} \mu\left( \frac{i}{2\ell}\right)\binom{\frac{d\ell}{k} - 1}{\ell - 1} \binom{N}{r}\cdot (r(N-r))^{\ell-1},\] and
\begin{align*}
    \theta_{j,k,d}&(\Gr(r,N))\\&=\frac{\varphi(j)}{2k}\sum_{i \mid \frac{k}{j}}\sum_{\substack{{\ell}\\{\ell \mid i\text{ and }k\mid d\ell}}} \mu\left( \frac{i}{\ell} \right) \binom{\frac{d\ell}{k} - 1}{\ell - 1}\left(\sum_{j = 0}^r \left(\binom{N}{j}-\binom{N}{j-1}\right)\left((r-j)(N-r-j)-j\right)^{\ell}\right).
\end{align*}
then we have
\begin{align*}
\overline{\mathsf{M}}_{1,r,N}^{nrt, \C^\star} &= \binom{N}{r}\left(\mathscr{A}_1 + \frac{\dot{\mathscr{A}}_0(\dot{\mathscr{A}}_0 + 1) + \frac{1}{4}\psi_2(\mathscr{A}_0'')}{1 - \psi_2(\mathscr{A}_0'')} - \frac{1}{2} \sum_{n \geq 1} \frac{\varphi(n)}{n}\log(1 - \psi_n(\mathscr{A}_0''))\right) \\&\phantom{space}+ \sum_{d \geq 2} q^d \sum_{k = 2}^{d} \eta_{k, d}(\Gr(r,N)) \frac{(1 + 2 \dot{\mathscr{A}}_0)^2}{(1 - \psi_2(\mathscr{A}_0''))^{k/2 + 1}} + \sum_{j \mid k} \theta_{j,k,d}(\Gr(r,N))\frac{1}{(1 - \psi_j(\mathscr{A}_0''))^{k/j}}.
\end{align*}
\end{prop}
Note that $\mathscr{A_0}$ and $\mathscr{A}_1$ have been determined by Getzler \cite{GetzlerGenusZero, GetzlerMHM}. Let us now specialize to topological Euler characteristics. Write \[\overline{\mathtt{M}}_{1, r, N} := \sum_{n, d} \chi^{\bbS_n}(\Mbar_{1, n}(\Gr(r, N), d))q^d \in \Lambda[\![q]\!].\] Define $\overline{\mathtt{M}}_{1, r, N}^{nrt}$ and $\mathtt{M}_{1, r, N}$ analogously for the no-rational-tails locus and locus of maps from smooth curves, respectively. Also set $\overline{\mathtt{M}}_{0, r, N}^{\star}$, $\overline{\mathtt{Q}}_{0, r, N}^{\ve,\star}$, and $\mathtt{Q}_{0, r, N}^{\ve,\star}$ for the corresponding specializations of $\overline{\mathsf{M}}_{0, r, N}^{\star}$, $\overline{\mathsf{Q}}_{0, r, N}^{\ve,\star}$, and $\mathsf{Q}_{0, r, N}^{\ve,\star}$, respectively. Let $B_{r,\ve}$ and $B_{r,0}$ be the $\bbL=1$ specializations of $\boldsymbol{B}_{r,\ve}$ and $\boldsymbol{B}_{r,0},$ namely \[{B}_{r, \ve}: p_j \mapsto p_j + \sum_{k = 1}^{M_{\ve}} \binom{r + k - 1}{k} \ q^{jk}\text{ and } {B}_{r, 0}: p_j \mapsto p_j + \prod_{i = 0}^{r-1} \frac{1}{1 - q^j} - 1. \]
Note that by Lemma \ref{lem:g0plethysm}, we have
\[ (p_1 - \mathtt{Q}^{\ve, \star}_{0, r, N}) \circ (p_1 + \overline{\mathtt{Q}}_{0, r, N}^{\ve, \star}) = p_1, \]
and by Corollary \ref{cor:nrtqvsm}, we have 
\[ B_{r, \ve}(\mathtt{M}_{0, r, N}^{\star}) = \mathtt{Q}^{\ve, \star}_{0, r, N}, \]
so $\overline{\mathtt{Q}}_{0, r, N}^{\ve, \star}$ can be calculated starting from Bagnarol's formula Theorem \ref{thm:bagnarol} for $\cat{M}_{0, r, N}$. In particular, together with Proposition \ref{prop:chi_genus_one}, the following result determines $\chi^{\bbS_n}(\Qbar_{1, n}^{\ve}(\Gr(r, N), d))$ for arbitrary $\ve, n, r, N$, and $d$.

\begin{cor}\label{cor:top_chi_genus_one}
    We have \[\overline{\mathtt{Q}}_{1, r, N}^{\ve} = B_{r,\ve}(\overline{\mathtt{M}}_{1, r, N}^{nrt})\circ (p_1 + \overline{\mathtt{Q}}^{\ve,\star}_{0,r,N}).\]
\end{cor}
\begin{proof}
    After specializing Lemma \ref{lem:nrtplethysm} and Corollary \ref{cor:nrtqvsm} to $\bbS_n$-equivariant Euler characteristics, it suffices to prove that $\chi^{\bbS_n}(\Mbar_{1,n}^{nrt}(\Gr(r,N),d)) = \chi^{\bbS_n}(\Mbar_{1,n}^{nrt}(\Gr(r,N),d)^{\C^\star}).$ This holds in general for (graded) $\bbS$-spaces admitting a commuting $\C^\star$-action \cite[Lemma 3.2]{ks-genus1}.
\end{proof}
\begin{rem}
    The strata of the $\C^\star$-fixed locus $\Qbar_{g,n}^{\ve}(\Gr(r,N),d)^{\C^\star}$ have been described in \cite[\S 5.1]{Toda} and \cite[\S 7.3]{MOP}. They involve the combinatorics of torsion sheaves, which in principle complicates the graph enumeration problem. From this perspective, the automorphisms ${B}_{r,\ve}$ and ${B}_{r, 0}$ encode the additional data relative to localization graphs for $\Mbar_{1, n}(\Gr(r, N), d)^{\C^\star}$ and can be potentially generalized to a technique for counting decorated graphs in other settings.
\end{rem}

\bibliographystyle{amsalpha}
\bibliography{reference}

\end{document}